\documentclass[12pt]{article}
\usepackage[utf8]{inputenc}   
\usepackage{amsmath}
\usepackage{amssymb}
\usepackage{amsthm}

\usepackage{tikz}
\usetikzlibrary{matrix,decorations,arrows}
\usetikzlibrary{decorations.markings}

\newcounter{remark}
\newcommand{\remark}{\addtocounter{remark}{1}
                       \par \quad {\bf \arabic{remark}}.\,
                      }
\newenvironment{rks}{\begin{quote}
                     \setcounter{remark}{0}
                     \setlength{\parskip}{0.25\parskip}
                     \renewcommand{\item}{\remark}
                     {\bf Remarks}
                     \nobreak
                    }
                    {\end{quote}}

\newenvironment{rk}{\begin{quote}
                     {{\bf Remark} --}
                    }{\end{quote}}

\newenvironment{preuve}{\medbreak \noindent {\bf Proof~---}}
                       {\hfill $\square$ \medbreak}

\newcommand{\Aff}{\mathbb A}
\newcommand{\NN}{\mathbb N}
\newcommand{\ZZ}{\mathbb Z}
\newcommand{\PP}{\mathbb P}

\newcommand{\RR}{\mathbb R}
\newcommand{\CC}{\mathbb C}
\newcommand{\FF}{\mathbb F}

\newcommand{\OO}{\mathcal O}

\newtheorem{theo}{Theorem}

\newtheorem{defi}[theo]{Definition}
\newtheorem{cor}[theo]{Corollary}
\newtheorem{lem}[theo]{Lemma}

\newtheorem{prop}[theo]{Proposition}

\newtheorem*{theointro}{Theorem}

\newcommand{\Div}{\operatorname{Div}}
\newcommand{\End}{\operatorname{End}}

\newcommand{\Id}{\operatorname{Id}}

\newcommand{\Jac}{\operatorname{Jac}}
\newcommand{\Ker}{\operatorname{Ker}}

\newcommand{\NS}{\operatorname{NS}}

\newcommand{\Proj}{\operatorname{Proj}}

\newcommand{\rang}{\operatorname{rk}}

\newcommand{\Sp}{\operatorname{Sp}}

\newcommand{\tr}{\operatorname{tr}}

\newcommand{\Cun}{X}
\newcommand{\Cdeux}{\Gamma}
\newcommand{\Graphe}{{\mathcal G}}
\newcommand{\Gsing}{{\mathcal G}_{\rm sing}}
\newcommand{\ga}{\gamma}

\newcommand{\Deltaa}{\pi'_2}\newcommand{\Deltab}{\pi'_1}

\newcommand{\vertiii}[1]{
    {\left\vert\kern-0.25ex\left\vert\kern-0.25ex\left\vert #1 
    \right\vert\kern-0.25ex\right\vert\kern-0.25ex\right\vert}}

\title{\bfseries Recursive towers of curves over finite fields using graph theory}
\author{
Emmanuel Hallouin \& Marc Perret\thanks{Institut de Math\'ematiques de Toulouse, UMR 5219}
}

\begin{document}
\maketitle

\begin{abstract}
We give a new way to study recursive towers of curves over a finite field,
defined {\em \`a la Elkies} from a bottom curve $\Cun$ and a correspondence $\Cdeux$
on $\Cun$.
A close examination of singularities leads to a necessary condition for a tower to be asymptotically good. Then, spectral theory on a directed graph, Perron-Frobenius theory and considerations on the class of $\Cdeux$ in $\NS (\Cun \times \Cun)$ lead to the fact that, under some mild assumption, a recursive tower can have in some sense only a restricted asymptotic quality. Results are applied to the Bezerra-Garcia-Stichtenoth tower along the paper for illustration.
\end{abstract}

\tableofcontents

\bigskip

\noindent AMS classification : 11G20, 14G05, 14G15, 14H20, 5C38, 5C50.

\noindent Keywords : Curves over a finite field, curves with many points, Graphs, Towers of function fields, Zeta functions.

\section*{Introduction}
Since Garcia-Stichtenoth's well-known Inventiones'95 paper
appeared  \cite{GSInvent},
many examples of {\em good recursive towers} of curves over finite
fields have been described in the literature. Recall
that a tower~${\mathcal T= \left(C_n\right)_{n \geq 1}}$ of smooth projective absolutely irreducible curves $C_n$ over a finite field ${\mathbb F}_q$ is said to be
{\em good} if it has many rational points over some finite extension of the base field. To be more precise, for any $r \geq 1$, denote  by
\begin{align*}
\lambda_r({\mathcal T})
&=
\lim_{n \rightarrow +\infty}
\frac{\text{number of points of~$C_n$ defined over~$\FF_{q^r}$}}
{\text{genus of~$C_n$}}, \\
\beta_r({\mathcal T})
&=
\lim_{n \rightarrow +\infty}
\frac{\text{number of points of~$C_n$ of degree~$r$}}
{\text{genus of~$C_n$}},
\end{align*}
(it turns out that these limits exist). The well-known Drinfeld-Vl{\u{a}}du{\c{t}} bound states that $\lambda_r({\mathcal T}) \leq \sqrt{q^r}-1$ for any $r \geq 1$ and any tower ${\mathcal T}$. The tower is said to be good if at least one $\lambda_r$ is non-zero. 
Even more precisely, the closer to zero is the deficiency (see equation $(\ref{deficiency})$ below), the better is the tower. Towers reaching the Drinfeld-Vl{\u{a}}du{\c{t}} bound over some finite extension of the base field have deficiency zero, hence are optimal.  One usualy denote by
$$
A(q) = \limsup_{g\to+\infty} \frac{N_q(g)}{g},
\quad \text{where}\quad
N_q(g) = \max_{\genfrac{}{}{0pt}{1}{\text{$X/\FF_q$ sm., proj.,}}{\text{ curve of genus~$g$}}}
\sharp X(\FF_q).
$$
Some recursive towers reach the Drinfeld-Vl{\u{a}}du{\c{t}} bound for $q$ square
\cite{GSInvent,GS_Degre2Moderee,Garcia},
others give interesting non-zero lower bounds for $A(q)$ for some non-square
values of $q$, as  for $A(q^3)$ \cite{Ap3-3,Ap3-1,Ap3-2}
or more recently for~$A(q^{2n+1})$, $n\geq 1$ \cite{Ap2nplus1}.
It turns out that all these towers are {\em recursive}
over the projective line~$\PP^1$:
they are given by an explicit correspondence~$\Cdeux$
on~${\mathbb P}^1$, and the curves of the towers are
---~sometimes irreducible components of~---
the normalizations of the curves
\begin{equation*}
C_n
=
\left\{
(P_1, \dots, P_n) \in \left({\mathbb P}^1\right)^n \mid
(P_i, P_{i+1}) \in \Cdeux,  \hbox{ for each } i = 1,\ldots, n-1
\right\}.
\end{equation*}
The point is that no author give the procedure they used to obtain, or merely to guess which explicit equation will lead to a good recursive tower.
It turns out that very few papers contain theoretical considerations on recursive towers. The first small family of exceptions are a series of papers from Elkies \cite{ElkiesET,ANTSV}, whose goal is to make it plausible that any good recursive tower should come from the modular world. Another very small family of exceptions are Lenstra's \cite{Lenstra} and the subsequent Beelen's \cite{Beelen} papers, who deal with possibilities of getting recursive towers with a great number of rational points. The last exception is Bouw and Beelen's
paper \cite{BeelenBouw}, who give a link between some good recursive towers and Picard-Fuchs differential equations in characteristic $p$. Up to our knowledge, these are the only theoretical studies of recursive towers. The reader is referred to the excellent survey of Li \cite{Li} for details.

\bigskip

The aim of this paper is twofold. We want to understand better which features of the data $(\Cun, \Cdeux)$ can lead to a good recursive tower, and to study up to which point a recursive tower can be good. The key ingredients are considerations on singular models of the tower, geometry of the surface $\Cun \times \Cun$ through the class of $\Cdeux$ in the Neron-Severi group $\NS(\Cun \times \Cun)$
and the introduction of a graph attached to a tower which permits us
to use some usual results in graph theory such as spectral theory
of adjacency matrices and Perron-Frobenius theory of non-negative matrices.

\medbreak

Section~\ref{s_RT} is only a set-up one. We fix notations, introduce the
standard invariants of a tower and state some common hypothesis for most
statements. The definition
of a recursive tower requires only a pair~$(\Cun,\Cdeux)$
where~$\Cun$ is a smooth,
projective, absolutely irreducible curve defined over~$\FF_q$ and
where~$\Cdeux$ is a correspondence on~$\Cun$
which is supposed to be absolutely
irreducible and reduced. In fact one can restrict ourselves to
correspondence of special type~$(d,d)$ for~$d\geq 2$
(see \S\ref{s_recursive_tower}, for precise definition).
By contrast, we do not need to restrict ourselves
to~$\Cun$ equal to the projective line~$\PP^1$.

\medbreak

In section~\ref{s_genus_sequences}, we focus on the geometry of a recursive
tower.
Most previous authors have
chosen the function field point of view. In doing so, an important part of the geometry of the tower
---~through the singularities of the models $C_n$ of the curves~---
disappears.
We investigate more closely this geometry. This leads us to
distinguish three models of towers: the {\em singular} one,
the {\em smooth} one, and finally the {\em sharp} one which is an avatar of the singular
model.
Of course, the smooth model ---~ which corresponds to the usual tower of
function fields~--- is the most interesting one.
At any stage, the three curves are birational. The sharp one being naturally
embedded in a smooth surface, we can evaluate by adjunction formula and desingularization the geometric genus sequence of the tower.
We deduce a first necessary condition for a tower to be good:
either the curves~$C_n$ are singular for any $n$ greater than some $n_0$,
or $g({\Cun}) \geq 2$ and both
projections~$\pi_i : \Cdeux\to\Cun$ for $i=1,2$ are \'etale over~$\Cun$
(proposition~\ref{SingOuEtale}).
More precisely, in the singular case, we evaluate how
the {\em global measure of singularity} should grow when~$n\to\infty$, for
a tower to be good. A key point for the rest of the paper
is the understanding
of the singular points. We characterize them, and we study other singular points on the intersection of the curves with
some hypersurfaces for later use: corollary~\ref{bouclesing} plays an important place in the main section~\ref{s_asymptotic}. 

\medbreak

In section~\ref{s_graph}, we associate to each recursive tower a {\em geometric} infinite
directed graph and for any $r \geq 1$, an {\em arithmetic} finite directed graph. They depend only on the base curve~$\Cun$ an
on the correspondence~$\Gamma$. These graphs are closely related,
but different, to the one introduced by Beelen \cite{Beelen}. The main
difference is that the former depend on the singular
model of the tower, while the later depends only on the smooth model, that is on the associated function fields tower.
Though we share some common observations with Beelen,
  we give some new applications of the
graph, especially in the last section~\ref{s_asymptotic}. It  is a
very convenient way to represent a tower ---~in some way better than the
equations themselves~---,
in the sense that some of the most important properties can be directly seen from it.
The degree, the singular points, the totally splitting points, sometimes
the irreducibility, can directly be read off the
graph. This will be illustrated on the  {\em BGS tower} over~$\FF_{p^3}$ (see equation (\ref{BGS-equation}) below) attaining the Zink's lower bound \cite{Ap3-3,Ap3-1,Ap3-2}.

\medbreak

Section~\ref{s_asymptotic} is the main one of this paper. It is devoted to the
asymptotic behaviour of a recursive tower.
Cycles of length~$n$ in our graph are in bijection with the points of~$C_{n+1}$
having equal first and last coordinates. We have thus two ways to count
them: the combinatorial one, involving adjacency matrix of a graph,
and the geometric one, involving intersection theory on the
surface~$\Cun\times\Cun$. The comparison of these countings, together with
a standard lemma of diophantine approximations and the previous study of
the singularities lead us to prove a strong constraint on the graph:

\begin{theointro}
Let~$(\Cun,\Cdeux)$ be a correspondence as in section~\ref{s_recursive_tower}
such that the curves~$C_n$ of the associated tower are all irreducible.
Then the graph~$\Graphe_\infty(\Cun,\Cdeux)$ has at most one finite $d$-regular strongly connected
component.
\end{theointro}

In order to deduce from this graph theoretical result some properties of recursive towers, we use Perron Frobenius theorem for non-negative
matrices as a last tool. This leads us to an accurate form of the connection between the spectral radius
of finite subgraphs of the geometric graph, connected components of these
subgraphs and the number of points on the singular model
(proposition~\ref{cle}).
Finally, all together, we prove our second necessary condition for a recursive
tower to be good (proposition~\ref{d-reg_ou_plein_de_pts_dans_desingularisation}). Especially we prove our main result about the~$(\beta_r)_{r\geq 1}$
sequence of a recursive tower:

\begin{theointro}
Let~$(\Cun,\Cdeux)$ be a correspondence as in section~\ref{s_recursive_tower}. Suppose that the curves~$C_n$ of the associated tower are all irreducible
and that the geometric genus sequence~$(g_n)_{n\geq 1}$ goes to~$+\infty$.  Suppose also that the singular points of $C_n$ give rise to a number of geometric points in $\widetilde{C}_n$ negligible compared to~$d^n$ for large $n$.
Then, there exists at most one integer $r \geq 1$ such that~$\beta_r \neq 0$.
\end{theointro}

Note that the hypotheses of this theorem are satisfied for a large part of the
known recursive good towers. This is the case of the BGS tower
(see loc. cit.). Combining with the closed formula for the geometric genus
of this tower, we are able to compute exactly ---~for the first time up to our knowledge~--- two invariants,
its defect~$\delta$ and its zeta function both defined
by Tsfasmann and Vl{\u{a}}du{\c{t}} \cite{TV02}.
Finally, several authors have recently studied some towers related to recursive one, having several non-vanishing $\beta_r$'s (e.g. Hesse, Stichtenoth and Tutdere \cite{HessStichTutdere}, or Ballet and Rolland \cite{BalletRolland}). It turns out that these towers are not recursive in the sense of this paper, but are pull back of a recursive one
under some finite surjective morphism $\pi : Y \rightarrow X$. We prove in the last section corollary \ref{basechange} that in this case, the number of non-zero $\beta_r$ is at most equal to the degree of $\pi$.

\section{Models of recursive towers}\label{s_RT}

\subsection{Invariants of towers of curves over a finite field}\label{s_invariants}

An irreducible {\em tower of curves}~${\mathcal T}$ over a finite
field~$\FF_q$ is a sequence of
absolutely irreducible curves $\left(C_n \right)_{n \geq 1}$
defined over~$\FF_q$ together with a
family of finite dominant morphisms $C_{n+1}~\rightarrow~C_{n}$.
For each~$n\geq 1$, let~$\widetilde{C}_n$ denote the normalization
of~$C_n$. Then~$\widetilde{{\mathcal T}} = (\widetilde{C}_n)_n$ is also
a tower of curves.

Our purpose is to study the following invariants of these towers.

\medbreak

\noindent {\bfseries At finite levels}, for~$n \geq 1$  and~$r \geq 1$:

\begin{itemize}
\item the {\em arithmetic genus}~$\ga_n = \ga(C_n)$ of~$C_n$ 
\item the common {\em geometric genus}~$g_n= g(C_n)$ of~$C_n$
and~$\widetilde{C_n}$;
\item the number~$N_r(\widetilde{C}_n) = \#\widetilde{C}_n(\FF_{q^r})$
of $\FF_{q^r}$-rational points of~$\widetilde{C}_n$;
\item the number~$B_r(\widetilde{C}_n)$ of points of~$\widetilde{C}_n$
of degree~$r$.
\end{itemize}

For any $n \geq 1$, we have~$g_n \leq \ga_n$,
and
\begin{equation}\label{N_r_B_r}
N_r(\widetilde{C}_n) = \sum_{d \mid r} r B_r(\widetilde{C}_n).
\end{equation}

\noindent {\bfseries Ultimately}, as usual, we introduce the two asymptotic invariants provided that the genus sequence satisfies $\lim_{n\to+\infty}g_n = +\infty$:
$$
\lambda_r({\mathcal T})
=
\lim_{n\to +\infty} \frac{N_r(\widetilde{C}_n)}{g_n} 
\quad \quad \text{and}\quad \quad
\beta_r({\mathcal T})
=
\lim_{n\to +\infty} \frac{B_r(\widetilde{C}_n)}{g_n}.
$$
Following Tsfasmann and Vl{\u{a}}du{\c{t}} \cite{TV02}, a tower is said to be {\em asymptotically exact} if these limits~$\lambda_r({\mathcal T})$ and $\beta_r({\mathcal T})$ do exist
for any $r \geq 1$, provided that the genus tends to infinity. Garcia and Stichtenoth have observed that a recursive tower is always asymptotically exact:

\begin{lem}\label{exacte}
Let ${\mathcal T} = (C_n)_{n \geq 1}$ be an irreducible tower of projective
smooth absolutely irreducible curves and let~$(d_n)_n$
denote the degrees sequence of the tower, i.e. $d_n = \deg(C_n\to C_1)$. Suppose
that~$\lim_{n\to+\infty} g(C_n) = +\infty$. 
Then, for any~$r\geq 1$,
the sequences~$\left(\frac{d_n}{g(C_n)}\right)_{n\geq 1}$
and~$\left(\frac{N_r(C_n)}{d_n}\right)_{n\geq 1}$
are convergent. The limits~$\lambda_r({\mathcal T})$
and~$\beta_r({\mathcal T})$ exist
and~$\lambda_r({\mathcal T})$ is non-zero if and only if
both sequences~$\left(\frac{g(C_n)}{d_n}\right)_{n\geq 1}$
and~$\left(\frac{N_r(C_n)}{d_n}\right)_{n\geq 1}$ admit a non-zero limit.
\end{lem}

\begin{preuve}
We compare the sequences~$(N_r(C_n))_n$ and~$(g(C_n)-1)_n$
with~$(d_n)_n$.
One has:
$$
N_r(C_n) \leq \deg(C_n\to C_{n-1})N_r(C_{n-1})
\qquad \text{and} \qquad
g(C_n)-1 \geq \deg(C_n\to C_{n-1})\left(g(C_{n-1}) - 1\right),
$$
the second inequality being a consequence of Riemann-Hurwitz formula.
Thus the sequence~$\left(\frac{N_r(C_n)}{d_n}\right)_n$ decreases,
while the sequence~$\left(\frac{g(C_n)-1}{d_n}\right)_n$ increases.
If moreover~$\lim_{n\to\infty} g(C_n) = +\infty$, we deduce
that sequences~$\left(\frac{N_r(C_n)}{d_n}\right)_n$,
$\left(\frac{d_n}{g(C_n)-1}\right)_n$
and~$\left(\frac{d_n}{g(C_n)}\right)_n$
are convergent.
This proves that the limit~$\lambda_r((C_n)_{n\geq 1})$ do exist.
By induction on $r$ thanks to the relation~(\ref{N_r_B_r}),
we also deduce that the limit~$\beta_r((C_n)_{n \geq 1})$ exist for
any~$r \geq 1$. The lemma follows.
\end{preuve}

From equation~(\ref{N_r_B_r}), we have for any $r \geq 1$
\begin{equation}\label{lambda_r_beta_r}
\lambda_r({\mathcal T}) = \sum_{d \mid r} d \beta_d({\mathcal T})
\end{equation}
and the important inequality
$$
A(q^r) \geq \lambda_r({\mathcal T}).
$$
A recursive tower is interesting only if at least
one~$\lambda_r$ exists and is non-zero, in which case the tower is said to
be {\em good}. One can be more precise. It has been proved by Tsfasman \cite{T92} that
$$\sum_{r=1}^{\infty} \frac{r\beta_r}{\sqrt{q^r}-1} \leq 1,$$
generalizing the well known Drinfeld-Vl{\u{a}}du{\c{t}} bound.
Tsfasmann and Vl{\u{a}}du{\c{t}} \cite{TV02} have also defined the {\em deficiency} of an asymptotically exact tower by
\begin{equation} \label{deficiency}
\delta({\mathcal T}) = 1-\sum_{r=1}^{\infty} \frac{r\beta_r}{\sqrt{q^r}-1} \in [0, 1].
\end{equation}
To sum up, a tower is good if $\delta <1$. It is said {\em optimal} if $\delta=0$.

\subsection{Recursive towers of curves over a finite field}\label{s_recursive_tower}

This article deals with specific towers of curves, the so-called {\em recursive} ones, defined by Elkies \cite{ElkiesET} as follows.

Let~$\Cun$ be a smooth projective absolutely irreducible algebraic curve of genus $g({\Cun}) \geq 0$ defined over the finite
field~$\FF_q$. Let~$\Cdeux$ be a correspondence {\em of type~$(d_1,d_2)$}
on~$ \Cun$; this means that
$d_1 = \Cdeux \cdot H$ and~$d_2 = \Cdeux \cdot V$,
where~$H = \Cun \times {\rm pt}$ and~$V = {\rm pt} \times \Cun$ denote the
horizontal and vertical divisors on~$\Cun\times\Cun$
(cf.~\cite[Chap~V,\S1,Ex 1.9, page~368]{Hartshorne}).

Consider the two
projection morphisms~$\pi_i : \Cun\times\Cun\to\Cun$, defined
by~$\pi_i(P_1,P_2) = P_i$ for $i=1,2$.
We have the following diagram:
$$
\begin{tikzpicture}[>=latex,baseline=(M.center)]
\matrix (M) [matrix of math nodes,row sep=0.75cm,column sep=0.75cm]
{
            &  |(XxX)| \Cun \times \Cun & \\
|(X1)| \Cun &                           & |(X2)| \Cun \\
};
\draw[->] (XxX) -- (X1) node[midway,anchor = south east] {$\pi_1$} ;
\draw[->] (XxX) -- (X2) node[midway,anchor = south west] {$\pi_2$} ;
\end{tikzpicture}
\qquad
\genfrac{}{}{0pt}{0}
{\text{which induces two}}
{\text{finite morphisms:}}
\qquad
\begin{tikzpicture}[>=latex,baseline=(M.center)]
\matrix (M) [matrix of math nodes,row sep=0.75cm,column sep=0.75cm]
{
            &  |(Gamma)| \Cdeux & \\
|(X1)| \Cun &                           & |(X2)| \Cun \\
};
\draw[->] (Gamma) -- (X1) node[midway,anchor = south east] {$\pi_1$} ;
\draw[->] (Gamma) -- (X2) node[midway,anchor = south west] {$\pi_2$} ;
\end{tikzpicture}
\quad
\begin{array}{c}
\deg(\pi_1) = d_2\\
\deg(\pi_2) = d_1
\end{array}
$$
Note that the curve~$\Cdeux$ is not supposed to be smooth. 
The irreducibility assumption of~$\Cdeux$ is natural since we will deal with irreducible towers.


\medskip

From these data, one can define three towers of curves.

\medbreak

$\bullet$ The {\em singular recursive tower} ${\mathcal T}(\Cun, \Cdeux)$ is
the sequence of curves~$\left(C_n\right)_{n \geq 1}$ defined by:
\begin{equation}\label{eq_def_C_n}
C_n =
\left\{
(P_1, P_2, \cdots, P_n) \in \Cun^n \mid
\text{$(P_i, P_{i+1}) \in \Cdeux$ for each~$i = 1, 2, \ldots, n-1$}
\right\}
\end{equation}
By definition, each curve~$C_n$ is embedded in the $n$-fold product~$\Cun^n$.
For~$1\leq i\leq n$, let~$\pi_i^n : C_n \to \Cun$ (or simply~$\pi_i$ if the
domain is clear from the context) be the $i$-th projection defined
by~$(P_1,\ldots,P_n) \longmapsto P_i$.

Thus~$C_1 = \Cun$ is supposed to be smooth, while~$C_2 = \Cdeux$
is not! So except for~$C_1$, the curves~$C_n$ for~$n \geq 2$ need
not to be smooth (even if $\Cdeux$ is, see
propositions~\ref{singular} and~\ref{lisse&totdec}).

\medbreak

$\bullet$ The {\em smooth recursive tower}  $\widetilde{\mathcal T}(\Cun, \Cdeux)$ is the sequence of smooth curves~$(\widetilde{C}_n)_{n \geq 1}$ where, for each~$n \geq 1$, we
denote by~$\widetilde{C}_n$ the normalization of the curve~$C_n$
and by~$\nu_n : \widetilde{C}_n\to C_n$ the desingularization morphism. 

\medbreak

$\bullet$ The {\em sharp recursive tower} ${\mathcal T}^{\sharp}(\Cun, \Cdeux)$ is the sequence of curves~$(C_n^{\sharp})_{n \geq 1}$, where~$C_n^\sharp$ is the pullback of the embedding~$\Cdeux \hookrightarrow \Cun\times\Cun$ along~$\pi_{n-1}^{n-1}\circ\nu_n \times \Id : \widetilde{C}_{n-1}\times\Cun \to \Cun\times\Cun$. It is also the pullback
of the embedding~$C_n \hookrightarrow C_{n-1}\times X$ along~$\nu_{n-1} \times \Id : \widetilde{C}_{n-1}\times\Cun \to C_{n-1}\times\Cun$,
 so that we have the cartesian diagram
$$
\begin{tikzpicture}[>=latex]
\matrix[matrix of math nodes,row sep=0.75cm,column sep=0.75cm]
{
|(Cdieze)| C_n^\sharp    & |(CtildeX)| \widetilde{C}_{n-1}\times\Cun \\
|(C)|      C_n         & |(CX)| C_{n-1}\times\Cun \\
|(Gamma)| \Gamma        & |(XX)| \Cun\times\Cun\\
};
\draw[->] (Cdieze) -- (C) ; \draw[->] (C) -- (Gamma) ;
\draw[->] (CtildeX) -- (CX) ; \draw[->] (CX) -- (XX) ;
\draw[right hook->] (Cdieze) -- (CtildeX) ;
\draw[right hook->] (C) -- (CX) ;
\draw[right hook->] (Gamma) -- (XX) ;
\end{tikzpicture}
$$

\medbreak

All the morphisms between the different curves are summarized in
figure~\ref{figure_morphismes}. All vertical maps and the two curved ones are finite morphisms, while the
straight diagonal maps are birational
isomorphisms if the $C_n$ are irreducible. Indeed, it is easily checked using the fiber product interpretation of
 $C_n^{\sharp}$ 
 that the map $\widetilde{C}_n \rightarrow C_n^{\sharp}$ is surjective, so that $C_n^{\sharp}$
 is also irreducible. Since moreover the composite map $\widetilde{C}_n \rightarrow C_n^{\sharp} \rightarrow C_n$ is a birational isomorphism and all curves are irreducible, both maps are birational isomorphisms. Therefore
$$
\ga(C_n) \geq \ga(C_n^\sharp) \geq g(\widetilde{C}_n),
$$
which means that the curve~$C_n^\sharp$ is singular, but less than $C_{n}$ itself.

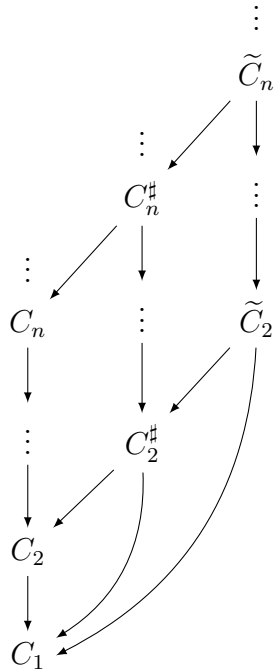
\begin{figure}
$$
\begin{tikzpicture}[>=latex]
\matrix[matrix of math nodes,row sep=0.75cm,column sep=0.75cm]
{
                &                         & |(Cntilde)| \widetilde{C}_n\\
                & |(Cndieze)| C_n^\sharp & |(v3)| \vdots \\
|(Cn)| C_n    & |(v2)| \vdots         & |(C2tilde)| \widetilde{C}_2 \\
|(v1)| \vdots & |(C2dieze)| C_2^\sharp & \\
|(C2)| C_2    & |(vide1)|              &\\
|(C1)| C_1    &                         & |(vide2)| \\
};
\draw[->] (Cn) -- (v1) ;
\draw[->] (v1) -- (C2) ;
\draw[->] (C2) -- (C1) ;
\draw[->] (Cndieze) -- (v2) ;
\draw[->] (v2) -- (C2dieze) ;
\draw[->] (Cntilde) -- (v3) ;
\draw[->] (v3) -- (C2tilde) ;
\draw[->] (Cntilde) -- (Cndieze) ;
\draw[->] (Cndieze) -- (Cn) ;
\draw[->] (C2tilde) -- (C2dieze) ;
\draw[->] (C2dieze) -- (C2) ;
\draw[->] (C2dieze) to [bend left] (C1) ;
\draw[->] (C2tilde) to [bend left] (C1.east) ;
\node[above,yshift=1em] at (Cn) {$\vdots$};
\node[above,yshift=1em] at (Cndieze) {$\vdots$};
\node[above,yshift=1em] at (Cntilde) {$\vdots$};
\end{tikzpicture}
$$
\caption{The three towers}\label{figure_morphismes}
\end{figure}

\medskip

In most of the examples studied in the literature, the base curve~$\Cun$ is the
projective line~$\PP^1$. As for the correspondence, it has most often separated
variables, that is:
$$
\Cdeux_{f,g}
=
\left\{
(P, Q) \in \PP^1 \times \PP^1 \mid f(P)=g(Q)
\right\}
$$
where~$f$ and~$g$ are two rational functions on~$\PP^1$. The curves~$C_n$
are then defined by:
$$
C_n
=
\{
(P_1,\ldots,P_n) \in (\PP^1)^n
\mid f(P_i) = g(P_{i+1}),\, i = 1,\ldots, n-1
\}.
$$

\medbreak

Even if, the {\em smooth recursive tower} is the most interesting
---~this is the tower studied by previous authors usually using the function
field language~---, we think that the consideration
of the {\em singular recursive tower} can be fruitful
due to its geometric definition. As for the {\em sharp recursive tower},
it turns to be useful in order to study the genus of the {\em smooth
recursive tower} using adjunction formula on smooth surfaces and normalization
process.

Since our final goal is to study the asymptotic behaviour of smooth absolutely irreducible curves, we will assume in most statements that the singular curves $C_n$ are irreducible. Up to our knowledge, the only important reducible recursive tower containing a good irreducible sub-tower is in \cite{Ap3-3}. However, this good irreducible sub-tower turned later to be itself recursive in \cite{Ap3-2}.
A simple criterion asserting this irreducibility, satisfied by most known good recursive towers, is given in the remark in section \ref{s_computation_genus}.

\medskip

For our purpose, under the assumption of irreducibility, we could restrict our
attention to the case of correspondence of type~$(d_1,d_2)$ with~$d_1 = d_2$
as this well known lemma shows.

\begin{lem} \label{d1d2}
Let~$(\Cun,\Cdeux)$ be a correspondence as in section~\ref{s_recursive_tower},
except that the type is assumed to be~$(d_1,d_2)$.
Let~${\mathcal T} = (C_n)_{n\geq 1}$ be the associated tower. Suppose that the
curves~$C_n$ are
irreducible for any~$n\geq 1$, and that the geometric genus
sequence~$(g_n)_{n\geq1}$ goes to infinity. If~$d_1\not= d_2$,
then $\lambda_r({\mathcal T}) = 0$ and $\beta_r({\mathcal T}) = 0$ for
any~$r\geq 1$.
\end{lem}

\begin{preuve}
Suppose for instance that~$d_1< d_2$, and let $r \geq 1$. Then one
has~$N_r(C_n) \leq N_r(C_1)d_1^{n-1}$. On the other hand, since the genus~$g(C_n)$
goes to infinity, one can also suppose
that~$g(C_1) \geq 2$ and by Hurwitz
genus formula, one has~$g(C_n) -1\geq d_2^{n-1}(g(C_1)-1)$ for any $n \geq 1$.
Therefore~$\lambda_r({\mathcal T}(\Cun,\Cdeux)) = 0$ since $\left(\frac{d_1}{d_2}\right)^n \to 0$. The assertion for the $\beta_r$'s follows by induction from
formula~(\ref{lambda_r_beta_r}).
\end{preuve}

In the whole paper we make the following assumptions:

\begin{quote}
{\bfseries Hypotheses~---~}{\em The curve~$\Cun$ is supposed to be smooth,
projective, absolutely irreducible, and defined over~$\FF_q$. The
correspondence~$\Cdeux$ on~$\Cun$ is supposed to be absolutely
irreducible, reduced, and of type~$(d,d)$ for~$d\geq 2$.}
\end{quote}

\section{Genus sequences in a recursive tower}\label{s_genus_sequences}

In order to compute the $\lambda_r$'s and~$\beta_r$'s invariants of a recursive tower, one needs to understand
the behaviour of the genus sequence. It turns out that $g_n$ and $\ga^{\sharp}_n$ are closely related thanks to adjunction formula (proposition~\ref{genre_arithmetic_geometric} thanks to lemma \ref{genre_tire_arriere}). This leads us to distinguish two kinds of recursive towers which could be good (proposition~\ref{SingOuEtale}). 
The proof of proposition~\ref{singular}, which gives a characterization of the singular points in a recursive tower, takes the largest part of this section. Then, we prove proposition~\ref{multiple} and its important corollary~\ref{bouclesing}, which is one of the tools in the proof of theorem~\ref{betar} in section~\ref{s_asymptotic}.

\subsection{Arithmetic versus geometric genus in recursive towers} 
\label{s_genre_geometrique_arithmetique}

Let~$(\Cun, \Cdeux)$ be as in section \ref{s_recursive_tower} and
consider~${\mathcal T}, {\mathcal T}^\sharp$ and $\widetilde{\mathcal T}$ the
associated towers of curves.
The sharp model turns here to be a useful tool to understand the geometric
genus sequence $(g_n)_{n\geq 1}$. We proceed in two steps: first we compare
the geometric genus~$g_n$ with the arithmetic genus~$\ga_n^\sharp$ using the adjunction
formula on the smooth surface $\widetilde{C}_{n-1}\times \Cun$, then
we prove an induction relation between~$g_n$ and~$g_{n-1}$ involving terms coming from desingularization of $C_n^{\sharp}$. 

\bigbreak

$\bullet$ The first step is classical. For any~$n\geq 2$, and
any~$P \in C_n^\sharp(\overline{\FF_q})$ be a geometric point, let~$\delta_P$ denote
the {\em measure of the singularity} at~$P$ (see Hartshorne
\cite{Hartshorne}, Chap~IV, Ex~1.8
or Liu \cite{Liu}, \S7.5),
that
is\footnote{Consistency would require sharp exponents for the
following~$\delta_P$, $\OO_P$ and~$\Delta_n$. For simplicity, we choose
to drop them.}
$$
\delta_P = \dim_{\overline{\FF_q}} \widetilde{\OO}_P / \OO_P,
$$
where~$\OO_P$ and~$\widetilde{\OO}_P$ denote the local ring
of~$C_n^\sharp$ at~$P$ and its integral closure. This measure is non-zero if
and only if the point~$P$ is singular so it makes sense to define
\begin{equation}\label{Delta_n}
\Delta_n = \sum_{P \in C_n^\sharp(\overline{\FF_p})} \delta_P
\end{equation}
as a measure of the whole singularities of $C_n^{\sharp}$.
Then the geometric and arithmetic genus of~$C_n^\sharp$ are related by
\begin{equation}\label{ga_g_Delta}
\ga_n^\sharp = g_n + \Delta_n
\end{equation}
(see loc. cit.). 

\medbreak

$\bullet$ Second, to prove the induction relation in Proposition \ref{genre_arithmetic_geometric}, \ref {ga_g},
 we need the following lemma.

\begin{lem} \label{genre_tire_arriere}
Let~$f_i : Y_i \to X_i$ be finite morphisms of smooth
absolutely irreducible projective curves of degree~$n_i$ for $i = 1,2$, and
let~$F : Y_1 \times Y_2 \to X_1 \times X_2$ be the
product morphism~$F = f_1 \times f_2$. If~$\Cdeux$ is a
correspondence
of type~$(d_1, d_2)$ between~$X_1$ and~$X_2$, then the arithmetic genus $\ga(F^*(\Cdeux))$
of the pull-back $F^*(\Cdeux)$ of~$\Cdeux$ by~$F$ is given by
$$
2\ga(F^*(\Cdeux)) - 2
=
n_1n_2 \Cdeux^2 + n_2d_2\left(2g(Y_1)-2\right) + n_1d_1\left(2g(Y_2)-2\right)
$$
where~$g(Y_i)$ denotes the genus of~$Y_i$ ($i=1,2$) and where~$\Cdeux^2$ is the
self-intersection of~$\Cdeux$ computed in the group~$\NS(X_1 \times X_2)$.
\end{lem}

\begin{preuve}
By adjunction formula (see Liu~\cite[theorem~1.37, page~390]{Liu}
in the geometric case page~376 therein),
the arithmetic genus is given by
$$
2 \ga\left(F^*(\Cdeux)\right) - 2
=
F^*(\Cdeux) \cdot \left(F^*(\Cdeux) + K_{Y_1 \times Y_2}\right),
$$
where~$K_{Y_1\times Y_2}$ is the canonical class in the Neron-Severi group
$\NS(Y_1 \times Y_2)$ of the smooth
surface~$Y_1 \times Y_2$. This class $K_{Y_1 \times Y_2}$ is known to be~$(2g(Y_2)-2)H + (2g(Y_1)-2)V$ where~$H$ and~$V$ denote the horizontal and
vertical classes in~$\NS(Y_1 \times Y_2)$ (see Hartshorne \cite{Hartshorne}, Chap~II, Ex~8.3). Then
\begin{align*}
2 \ga\left(F^*(\Cdeux)\right) - 2
&=
F^*(\Cdeux) \cdot F^*(\Cdeux)
+
(2g(Y_2)-2) F^*(\Cdeux) \cdot H
+ (2g(Y_1)-2) F^*(\Cdeux) \cdot V.
\end{align*}
Denote by~$h$ and~$v$ the horizontal and
vertical classes in~$\NS(X_1 \times X_2)$. By the projection formula
(see Liu~\cite[Theorem~2.12, page~398]{Liu}), we have
\begin{align*}
F^*(\Cdeux) \cdot F^*(\Cdeux)
&=
\Cdeux \cdot F_* F^*(\Cdeux) = \Cdeux \cdot \deg(F) \Cdeux = n_1n_2\Cdeux^2,\\
F^*(\Cdeux) \cdot H
&=
\Cdeux \cdot F_*(H) = \Cdeux \cdot n_1 h = n_1 d_1, \\
F^*(\Cdeux) \cdot V
&=
\Cdeux \cdot F_*(V) = \Cdeux \cdot n_2 v = n_2 d_2
\end{align*}
since we have~$d_1 = \Gamma\cdot h$
and~$d_2 = \Gamma \cdot v$ by the very definition of the type~$(d_1,d_2)$.
\end{preuve}

\begin{prop}\label{genre_arithmetic_geometric}
Let~$(\Cun,\Cdeux)$ be a correspondence as in section~\ref{s_recursive_tower}.
Let~$(g_n)_{n\geq 1}$ and~$(\ga_n^\sharp)_{n\geq 1}$ be the geometric and
sharp-arithmetic genus sequence of the associated tower.
\begin{enumerate}
\item\label{ga_g} The sharp-arithmetic genus~$\ga_n^\sharp$ and the geometric genus~$g_{n-1}$ are
related, for $n \geq 2$, by
$$
\ga_n^\sharp - 1 = d(g_{n-1}-1) + d^{n-2}\left[(\ga^\sharp_2-1) - d(g_1 - 1)\right].
$$
\item\label{gn} For any~$n\geq 1$, the geometric genus~$g_n$ is given by
$$
g_n -1
=
\begin{cases}
(n-1)d^{n-2}\left[(\gamma_2-1)-d(g_1-1)\right] + d^{n-1}(g_1-1) - \sum_{i=2}^n d^{n-i}\Delta_i &\text{(general case)}\\
(n-1)d^{n-2}\left[(g_2-1)-d(g_1-1)\right] + d^{n-1}(g_1-1) & \text{(smooth case)}
\end{cases}
$$
where the~$\Delta_i$'s are defined in formula~$(\ref{Delta_n})$ and
where~$\ga_2$ denote the arithmetic genus of~$\Cdeux$.
\end{enumerate}
\end{prop}

\begin{preuve}
To prove~$\ref{ga_g}$, we first apply lemma~\ref{genre_tire_arriere}
with~$Y_1 = \widetilde{C}_{n-1}$,
$Y_2 = \Cun$,~$X_1 = X_2 = \Cun$,
$f_1 = \pi_{n-1}^{n-1}\circ\nu_{n-1}$ (see~\S\ref{s_recursive_tower} for
definitions) and~$f_2 = \Id$. We
get~$n_1 = d^{n-2}$, $n_2 = 1$ and
$$
2\ga_n^\sharp - 2 = d^{n-2}\Gamma^2 + d(2g_{n-1}-2) + d^{n-1}(2g_1 - 2).
$$
In particular, for~$n=2$, this leads to~$\Gamma^2 = (2\gamma_2-2) - 2d(2g_1-2)$.
Substituting this expression of~$\Gamma^2$ in the preceding equation permits to
conclude.

To prove~$\ref{gn}$,
let~$u_n = \frac{g_n -1}{d^n}$. From $\ref{ga_g}$ together with (\ref{ga_g_Delta}), we deduce
the induction relation
$$
u_n = u_{n-1} + \frac{(g_2-1) - d(g_1-1) + \Delta_2}{d^2} - \frac{\Delta_n}{d^n}.
$$
An easy calculation gives the general formula. If all the $C_n$ are smooth, then all~$\Delta_i$ vanish and $\gamma_2=g_2$.
\end{preuve}

\subsection{Another necessary condition for a tower to be good}

We prove that under the irreducibility assumption, 
the tower needs either to be singular, or to be constructed from an \'etale correspondence on a curve $X$ in order to be good.

\begin{prop}\label{SingOuEtale}
Let~$(\Cun,\Cdeux)$ be a correspondence as in section~\ref{s_recursive_tower} and
let~${\mathcal T} = (C_n)_{n\geq 1}$ be the associated recursive tower.
Suppose that~$C_n$ is irreducible for any~$n\geq 1$ and that the genus
sequence~$(g_n)_{n\geq 1}$ tends to infinity. If there is at least one $r \geq 1$ such that $\lambda_r({\mathcal T}) >0$, then
\begin{enumerate}
\item either $C_n$ is singular for any $n$ greater than some $n_0$;

\item or $g_1=g({\Cun}) \geq 2$ and both covers~$\pi_i : \Cdeux\to\Cun$ for $i=1,2$ are \'etale over~$\Cun$.
\end{enumerate}
\end{prop}

\begin{preuve} Suppose that $C_n$ is smooth for any $n \geq 1$. Then by the last item of proposition~\ref{genre_arithmetic_geometric}, 
one obtains for any $n \geq 1$
$$
g_n =
(n-1)d^{n-2}\left[(g_2-1) - d(g_1-1)\right] + d^{n-1}(g_1-1) + 1.
$$
If~$(g_2-1) - d(g_1-1) \not= 0$, then
$$
g_n \sim (n-1)\left((g_2-1) - d(g_1-1)\right)d^{n-2}.
$$
On the other hand, using the projection morphism from~$C_n$ to~$C_1$
given by~$(P_1,\ldots,P_n) \mapsto P_1$ of degree~$d^{n-1}$, one
deduces that
$$
N_r(C_n) \leq N_r(C_1) \times d^{n-1}
$$
for any $r \geq 1$. Therefore,
$$
\frac{N_r(C_n)}{g_n}
\leq
\frac{N_r(C_1)d^{n-1}}{g_n}
\sim
\frac{N_r(C_1)}{(g_2-1)-d(g_1-1)} \times \frac{d^{n-1}}{(n-1)d^{n-2}}
\underset{n\to+\infty}{\longrightarrow} 0
$$
and~$\lambda_r({\mathcal T}(\Cun,\Cdeux)) = 0$.

If~$(g_2-1) - d(g_1-1) = 0$, then both projections must be \'etale and one
must have~$g_1 = g(\Cun) \geq 1$. Finally if~$\Cun$ is an elliptic curve, then  Riemann-Hurwitz yields to $g_n=1$ for any $n$, which doesn't tend to infinity.
\end{preuve}

It worth to noticing that if both morphisms are non-\'etale, then not only the tower~$(C_n)_n$ need to be singular, but it needs to be {\em sufficiently singular}.
More precisely, by lemma~\ref{exacte},
the genus sequence must behave like~$c \times d^n$ for
some~$c > 0$. By proposition~\ref{genre_arithmetic_geometric},
the measures of singularities must be large
enough to satisfy:
$$
\sum_{i=2}^n d^{n-i} \Delta_i
=
(n-1)d^{n-2}\left[(\gamma_2-1)-d(g_1-1)\right] + c' \times d^{n}
+ o(d^n)
$$
for some~$c' > 0$.

\medbreak

This proposition~\ref{SingOuEtale} motivates a more accurate study of singular points of $C_n$. This is the aim of the next section.

\subsection{Singular points of $C_n$} \label{SectionSingular}

The goal of this section is twofold. First, we characterize the singular points of the
curves~$C_n$ in proposition~\ref{singular}. Then we prove corollary \ref{bouclesing} about the singularities of cycles which will be a key point later; it will be responsible of the crucial ``$-1$" at the end of formula (\ref{moinsun}).

\bigbreak

To begin with, let~$X$ and~$Y$ be two 
projective absolutely
irreducible curves over\footnote{In fact, a large part of this section works over an arbitrary field $k$.}~$\FF_q$ and let~$\Gamma$ be a correspondence between~$X$ and~$Y$,
without any vertical, nor horizontal, components. We still denote
by~$\pi_1$ and~$\pi_2$ the projections onto the first and second factors.
Let~$P \in X$ and~$Y \in Q$ be geometric {\em smooth} points
such that~$(P,Q) \in \Gamma$. Consider affine neighborhoods
of~$P\in U \subset \Aff^r$,~$Q\in V \subset \Aff^s$ and~$(P,Q)\in W \subset \Aff^{r+s}$. Suppose that the two affine curves~$U$ and~$V$
are respectively defined by~$(r-1)+\rho$ and~$(s-1) + \sigma$ equations
(where~$r,s\geq 1$ and~$\rho,\sigma\geq 0$); suppose also that besides the
equations of~$U$ and~$V$, we need~$1+\tau$ more equations to
define~$W$ (where~$\tau\geq 0$).
Taking into account that the equations defining~$U$ (resp.~$V$) only depend
on the $r$ first (resp. $s$ last) indeterminates, we deduce that the
jacobian matrix of the point~$(P,Q)\in W$ has the following shape:
\begin{equation}\label{J_Gamma}
J_{\Gamma}(P,Q)
=
\begin{tikzpicture}[baseline=(M.center)]
\matrix (M) [matrix of math nodes,left delimiter=(,right delimiter=),
             nodes in empty cells,
             ]
{
&        & & &        & \\
&        & & &        & \\
& J_X(P) & & &        & \\
&        & & &        & \\
&        & & &        & \\
&        & & &        & \\
& A      & & & B      & \\
&        & & &        & \\
&        & & &        & \\
&        & & &        & \\
&        & & & J_Y(Q) & \\
&        & & &        & \\
&        & & &        & \\
};
\draw[rounded corners] (M-1-1.north west) rectangle (M-5-3.south east);
\draw[rounded corners] (M-9-4.north west) rectangle (M-13-6.south east);
\draw[rounded corners] (M-6-1.north west) rectangle (M-8-3.south east);
\draw[rounded corners] (M-6-4.north west) rectangle (M-8-6.south east);
\end{tikzpicture}
\end{equation}
where~$J_X(P)$ and~$J_Y(Q)$ denote the jacobian matrices of~$X$ at~$P$
and~$Y$ at~$Q$.
Since the curves~$X$ and~$Y$ are supposed to be smooth at~$P$ and~$Q$,
the jacobian submatrices
$J_X(P)$ and~$J_Y(Q)$ have rank
equal to~$(r-1)$ and~$(s-1)$ respectively.
Therefore, due to its shape, $J_\Gamma(P,Q)$ has rank greater or equal
to~$r+s-2$. On the other hand,~$J_\Gamma(P,Q)$ has rank less
or equal to~$r+s-1$ since~$\Gamma$ is a curve locally embedded in~$\Aff^{r+s}$. 
We easily deduce that
\begin{align}
\rang(J_\Gamma(P,Q)) = r+s-2
\quad&\Longleftrightarrow\quad
\rang\begin{pmatrix}J_X(P)\\A\end{pmatrix} = r-1
\;\text{and}\;
\rang\begin{pmatrix}B\\J_Y(Q)\end{pmatrix} = s-1 \label{rg_min}
\end{align}
and thus that
\begin{align}
\rang(J_\Gamma(P,Q)) = r+s-1
\quad&\Longleftrightarrow\quad
\rang\begin{pmatrix}J_X(P)\\A\end{pmatrix} = r
\;\text{or}\;
\rang\begin{pmatrix}B\\J_Y(Q)\end{pmatrix} = s \label{rg_max}
\end{align}

\medbreak

{\bfseries Study of the smoothness of~$\Gamma$ at $(P, Q)$.} The point $(P, Q)$ is smooth if and only if the jacobian matrix has maximal rank~$r+s-1$, that is
by~(\ref{rg_max}):
\begin{align}
\genfrac{}{}{0pt}{0}
{\text{the point~$(P,Q) \in \Gamma$}}
{\text{is {\bfseries smooth}}}
\quad&\Longleftrightarrow\quad
\rang\begin{pmatrix}J_X(P)\\A\end{pmatrix} = r
\;\text{or}\;
\rang\begin{pmatrix}B\\J_Y(Q)\end{pmatrix} = s \label{pt_lisse}
\end{align}
In that case, one can
extract from~$J_X(P)$ an $(r-1)\times r$ block~$J'_X(P)$ of maximal rank,
from~$J_Y(Q)$ an~$(s-1)\times s$ block~$J'_Y(Q)$ of maximal rank and from
the "correspondence" block $A B$
exactly one line such that
the $(r+s-1)\times (r+s)$ matrix
$$
J'_\Gamma(P,Q)
=
\begin{tikzpicture}[baseline=(M.center)]
\matrix (M) [matrix of math nodes,left delimiter=(,right delimiter=),
             nodes in empty cells,
             ]
{
&        & & &        & \\
& J'_X(P) & & &        & \\
&        & & &        & \\
a_1 &    & a_r & b_1 &  & b_s \\
&        & & &        & \\
&        & & & J'_Y(Q) & \\
&        & & &        & \\
};
\draw[rounded corners] (M-1-1.north west) rectangle (M-3-3.south east);
\draw[rounded corners] (M-5-4.north west) rectangle (M-7-6.south east);
\draw[loosely dotted,thick] (M-4-1.east)--(M-4-3.west);
\draw[loosely dotted,thick] (M-4-4.east)--(M-4-6.west);
\end{tikzpicture}
$$
has maximal rank equal to~$(r+s-1)$. The rank of each block does not depend
on the choice of the line in the correspondence block~$AB$ since one must have
$$
\rang
\begin{pmatrix}&J'_X(P)&\\a_1&\cdots&a_r\end{pmatrix}
=
\rang\begin{pmatrix}J_X(P)\\A\end{pmatrix}
\quad\text{and}\quad
\rang
\begin{pmatrix}b_1&\cdots&b_s\\&J'_Y(Q)&\end{pmatrix}
=
\rang\begin{pmatrix}B\\J_Y(Q)\end{pmatrix}.
$$
(in fact the vector space generated by the lines of~$J_X(P)$
and~$(a_1,\ldots,a_r)$ must be equal to the vector space generated by the lines
of~$J_X(P)$ and the lines of~$A$).
Moreover the minors of maximal size of~$J'_\Gamma(P,Q)$ can
be easily described in terms of the minors of maximal size
of~$J'_X(P)$ and~$J'_Y(Q)$. More precisely,
let~$\delta_1(J'_X(P)),\ldots,\delta_r(J'_X(P))$
and~$\delta_1(J'_Y(Q)), \ldots, \delta_s(J'_Y(Q))$ be the minors of
maximal size of~$J'_X(P)$ and~$J'_Y(Q)$ (listed with alternate signs),
and define
\begin{equation} \label{Deltaa}
\Deltaa(P,Q) = \begin{vmatrix}&J'_X(P)&\\a_1&\cdots&a_r\end{vmatrix}
\qquad\text{and}\qquad
\Deltab(P,Q) = \begin{vmatrix}b_1&\cdots&b_s\\&J'_Y(Q)&\end{vmatrix}.
\end{equation}
Then the minors of maximal size of~$J'_\Gamma(P,Q)$
are the~$\delta_i(J'_X(P))\Deltab(P,Q)$ for~$1\leq i\leq r$ and
the~$\delta_j(J'_Y(Q))\Deltaa(P,Q)$ for~$1\leq j\leq s$.

Recall that if~$M \in M_{n-1,n}(k)$ is a
matrix of rank~$n-1$, then its kernel is generated by
the vector whose coordinates are its minors of size~$(n-1)$ with alternate
signs. Therefore, the kernel of~$J'_\Gamma(P,Q)$,
which is nothing else than the tangent line of~$\Gamma$
at~$(P,Q)$, is thus generated by the
vector
$$
\begin{pmatrix}
\delta_1(J'_X(P)) \Deltab(P,Q)\\
\vdots\\
\delta_r(J'_X(P)) \Deltab(P,Q)\\
\delta_1(J'_Y(Q)) \Deltaa(P,Q)\\
\vdots\\
\delta_s(J'_Y(Q)) \Deltaa(P,Q)\\
\end{pmatrix}.
$$
 Because at least one of the minors
of~$J'_X(P)$ and of~$J'_Y(Q)$ are non zero by smoothness of~$X$ and~$Y$
at~$P$ and~$Q$, this vector is non zero if and only if $\Deltaa(P,Q)$
or~$\Deltab(P,Q)$ are non zero, that is
if and only if~$ \Gamma$ is smooth at $(P, Q)$ by (\ref{pt_lisse}).

\medbreak

Note for later use that, independently of the choice of the line inside
the block $AB$, one has
\begin{equation}\label{pi_prime_nul}
\Deltaa(P,Q) = 0
\;\Leftrightarrow\;
\rang\begin{pmatrix}J_X(P)\\A\end{pmatrix} = r-1
\quad\text{and}\quad
\Deltab(P,Q) = 0
\;\Leftrightarrow\;
\rang\begin{pmatrix}B\\J_Y(Q)\end{pmatrix} = s-1.
\end{equation}

\medbreak

{\bfseries Study of \'etaleness of $\pi_i$ at~$(P,Q) \in \Gamma$.}
Since~$Q$ is a smooth point of $Y$, the projection~$\pi_2 : \Gamma\to Y$ is \'etale
at~$(P,Q)$ if and only if the point~$(P,Q)\in\Gamma$ is smooth and the
induced map~$\Ker(J_\Gamma(P,Q)) \to \Ker(J_Y(Q))$
is an isomorphism, that is non
zero since kernels are lines here. In view of the generator of the tangent
line at~$(P,Q)$, this application is non zero if and only if~$\Deltaa(P,Q)$
is non zero because at least one of the minors of~$J'_Y(Q)$ is non zero.
Thanks to~$(\ref{pi_prime_nul})$, we deduce that
\begin{equation}\label{pi_2_etale}
\genfrac{}{}{0pt}{0}
        {\text{$\pi_2$ is \'etale}}{\text{at~$(P,Q)$}}
\quad\Leftrightarrow\quad
\rang(J_\Gamma(P,Q)) = r+s-1
\;\text{and}\;
\rang\begin{pmatrix}J_X(P) \\ A\end{pmatrix} = r
\quad\overset{\text{by } (\ref{rg_max})}{\Leftrightarrow}\quad
\rang\begin{pmatrix}J_X(P) \\ A\end{pmatrix} = r.
\end{equation}
Of course, we also prove the same way that
\begin{equation}\label{pi_1_etale}
\genfrac{}{}{0pt}{0}
        {\text{$\pi_1$ \'etale}}{\text{at~$(P,Q)$}}
\quad\Leftrightarrow\quad
\rang(J_\Gamma(P,Q)) = r+s-1
\;\text{and}\;
\rang\begin{pmatrix}B\\J_Y(Q)\end{pmatrix} = s
\quad\overset{\text{by } (\ref{rg_max})}\Leftrightarrow\quad
\rang\begin{pmatrix}B\\J_Y(Q)\end{pmatrix} = s.
\end{equation}

From the negation of~(\ref{pt_lisse}), and from~(\ref{pi_2_etale}) and~(\ref{pi_1_etale}), we deduce
that the singularity of~$(P,Q)$ in $\Gamma$ can be characterized only
using \'etaleness by
\begin{equation}\label{pt_sing_not_etale}
\genfrac{}{}{0pt}{0}
{\text{the point~$(P,Q)\in\Gamma$}}
{\text{is {\bfseries singular}}}
\quad\Longleftrightarrow\quad
\genfrac{}{}{0pt}{0}
{\text{both projections~$\pi_1$ and~$\pi_2$}}
{\text{are {\bfseries not} \'etale at~$(P,Q)$}}.
\end{equation}
In the following proposition, we prove that such a characterization still
occurs for the curves~$C_n$ of a recursive tower.

\begin{prop} \label{singular}
Let~$(\Cun,\Cdeux)$ be a correspondence as in section~\ref{s_recursive_tower} and
let~$(C_n)_{n\geq 1}$ be the associated recursive tower.
A point~$(P_1,\ldots,P_n)\in C_n$ is singular if and only if there
exist~$1\leq i \leq j < n$ such that~$\pi_2$ is not \'etale at~$(P_i,P_{i+1})$
and~$\pi_1$ is not \'etale at~$(P_j,P_{j+1})$.
\end{prop}

\begin{preuve}
One can suppose that the points~$P_1,\ldots,P_n\in \Cun$ are contained in the
same affine open subspace $U \subset \Aff^r$, and that the affine
curve~$U$ is
defined by~$(r-1)+\rho$ equations. Suppose also
that, locally in~$U\times U \subset \Aff^{2r}$, $\Gamma$ is defined, in addition of the equations coming from $U$, by~$1+\tau$ equations
in~$\Aff^r\times\Aff^r=\Aff^{2r}$. Thus~$C_n$ is locally embedded
in~$\Aff^{nr}$ and the jacobian matrix of~$C_n$ at~$(P_1,\ldots,P_n)$
is an $\left((nr-1) + n\rho + (n-1)\tau\right)\times nr$ matrix
looking like
$$
J_{C_n}(P_1,\ldots,P_n)
=
\begin{tikzpicture}[baseline=(M.center)]
\matrix (M) [matrix of math nodes,left delimiter=(,right delimiter=),
             nodes in empty cells,
             minimum width = 4em,
             ]
{
J_\Cun(P_1) &            &          &        & \\
A_1        & B_1        &          &        & \\
           & J_\Cun(P_2) &          &        & \\
           &            & \ddots   &        & \\
           &            &          & A_{n-1} & B_{n-1} \\
           &            & &        & J_\Cun(P_n) \\
};
\draw[rounded corners] (M-1-1.north west) rectangle (M-1-1.south east);
\draw[rounded corners] (M-2-1.north west) rectangle (M-2-1.south east);
\draw[rounded corners] (M-2-2.north west) rectangle (M-2-2.south east);
\draw[rounded corners] (M-3-2.north west) rectangle (M-3-2.south east);
\draw[rounded corners] (M-5-4.north west) rectangle (M-5-4.south east);
\draw[rounded corners] (M-5-5.north west) rectangle (M-5-5.south east);
\draw[rounded corners] (M-6-5.north west) rectangle (M-6-5.south east);
\end{tikzpicture}
$$
Since every~$P_i\in X$ is smooth, every "jacobian" block has rank equal
to~$(r-1)$. Since~$\Cdeux$ is locally a curve in~$\Aff^{2r}$, every block
looking like the jacobian matrix of equation~(\ref{J_Gamma}) has rank less
than or equal to~$(2r-1)$ (and greater than or equal to~$(2r-2)$).
Hence the rank
of~$J_{C_n}(P_1,\ldots,P_n)$ is greater
than or equal to~$n(r-1)$ (contributions of the "jacobian" blocks) and every
"correspondence"-blocks has a contribution to the whole rank of at most~$1$.

A point~$(P_1,\ldots,P_n)\in C_n$ is smooth if and only if the rank
of~$J_{C_n}(P_1,\ldots,P_n)$ is equal to~$nr-1 = n(r-1) + (n-1)$. This is
plausible only if each of the $(n-1)$ "correspondence"-blocks has a contribution
to the whole rank exactly equal to~$1$. Conversely, a
point~$(P_1,\ldots,P_n) \in C_n$ is singular if and only if there exists at
least one "correspondence"-block whose lines are all in the vector space
generated by the remaining lines. In particular, if~$(P_1,\ldots,P_m)\in C_m$
is singular then so is~$(P_1,\ldots,P_m,P_{m+1},\ldots,P_n)\in C_n$ for
every~$n\geq m$. Then~$(P_1,\ldots,P_n) \in C_n$ is singular if and only
if there exists~$1\leq j\leq n-1$ such that~$(P_1,\ldots,P_j) \in C_j$ is smooth
while~$(P_1,\ldots,P_{j+1})\in C_{j+1} \subset C_j\times \Cun$ is singular.
Applying the beginning of this section to the correspondence~$C_{j+1}$
on the product~$C_j\times \Cun$ of smooth curves, we deduce that
this is equivalent to
$$
\rang
\begin{tikzpicture}[baseline=(M.center)]
\matrix (M) [matrix of math nodes,left delimiter=(,right delimiter=),
             nodes in empty cells,
             minimum width = 2.5em,
             ]
{
&        & \\
&        & \\
&        & \\
 &    & A_j \\
};
\draw[rounded corners] (M-1-1.north west) rectangle (M-3-3.south east);
\draw[rounded corners] (M-4-3.north west) rectangle (M-4-3.south east);
\node at (M-2-2) {$J_{C_j}(P_1,\ldots,P_j)$};
\end{tikzpicture}
 = rj-1
\qquad\text{and}\qquad
\rang\begin{pmatrix}B_j\\J_X(P_{j+1})\end{pmatrix} = r-1,
$$
where~$A_j$ and~$B_j$ denote the two blocks coming from the
condition~$(P_j, P_{j+1})\in\Cdeux$. By the negation of~(\ref{pi_1_etale}),
the second
equality occurs if and
only if the projection~$\pi_1$ is not \'etale at~$(P_j,P_{j+1})$.
As to the first equality, it occurs if and only if there
exists~$1\leq i \leq j$ such that~$\pi_2$ is not \'etale at~$(P_i,P_{i+1})$,
that is
$$
\exists 1\leq i\leq j, \qquad
\rang\begin{pmatrix}J_X(P_{i})\\A_i\end{pmatrix} = r-1.
$$
Indeed, since~$(P_1,\ldots,P_j)\in C_j$ is smooth,
the jacobian matrix~$J_{C_j}(P_1,\ldots,P_j)$ has a rank equal
to~$rj-1$. As in the case of a correspondence on a product of only two
curves, one can extract $(r-1)$-rank blocks~$J'_X(P_1), \ldots, J'_X(P_j)$
from the "jacobian" blocks~$J_X(P_1), \ldots, J_X(P_j)$, and
lines~$(a_{k,1},\ldots,a_{k,r},b_{k,1},\ldots,b_{k,r})$
from the "correspondence'' block~$(A_k,B_k)$ for~$1\leq k\leq j-1$, to obtain a
$(rj-1) \times rj$ matrix~$J'_{C_j}(P_1,\ldots,P_j)$ of maximal rank.
The first rank equality is thus equivalent to the fact
that for every choice of line~$(a_{j,1},\ldots,a_{j,r})$ in the block~$A_j$,
this line must be a linear combination of the lines
of~$J'_{C_j}(P_1,\ldots,P_j)$. So one must have
$$
\begin{vmatrix}&J'_X(P_1)&\\a_{1,1}&\cdots&a_{1,r}\end{vmatrix}
\times\cdots\times
\begin{vmatrix}&J'_X(P_{j-1})&\\a_{j-1,1}&\cdots&a_{j-1,r}\end{vmatrix}
\times
\begin{vmatrix}&J'_X(P_j)&\\a_{j,1}&\cdots&a_{j,r}\end{vmatrix} = 0.
$$
Hence,
$$
\text{either } \rang\begin{pmatrix}J_X(P_j)\\A_j\end{pmatrix}=r-1
\quad\text{or}\quad
\exists 1\leq i\leq j-1, \;
\begin{vmatrix}&J'_X(P_{i})&\\a_{i,1}&\cdots&a_{i,r}\end{vmatrix} = 0.
$$
But, by~$(\ref{pi_prime_nul})$, the last case is equivalent to
$$
\rang\begin{pmatrix}J_X(P_i)\\A_i\end{pmatrix}=r-1.
$$
In both cases, by the negation of~(\ref{pi_2_etale}), we conclude that there
exists~$1\leq i\leq j$ such that the projection~$\pi_2$ is not \'etale
at~$(P_i,P_{i+1})$.
\end{preuve}

\begin{rk}
This is a particular feature of recursive towers. It doesn't hold true in general that $(P, Q)$ is singular in the pullback curve $\left(\pi \times \Id\right)^*(\Cdeux)$ for any correspondence $\Cdeux$ on $\Cun$, any morphism $\pi : Y \to \Cun$   from a singular curve $Y$ (here $Y = C_{m}$) and any singular point $P$ on $Y$ such that $(\pi(P), Q) \in \Cdeux$. 
\end{rk}

In the case of correspondences of type $\Cdeux_{f, g}$ on $X = {\mathbb P}^1$ so widely used in the literature,
one can take~$r = 1$
and~$\rho=\tau=0$ in the above proof. The characterization of the singular points becomes:

\begin{cor}
Let~$\Cdeux_{f,g}$ be a correspondence on~$\PP^1$
where~$f$ and~$g$ are two non constant functions on~$\PP^1$
and let~$(C_n)_{n\geq 1}$
be the corresponding singular recursive tower. A
point~$(P_1,\ldots,P_n)\in C_n$ is singular if and only if there
exist~$1\leq i< j\leq n$ such that~$P_i$ is a ramified point of~$f$
and~$P_j$ is a ramified point of~$g$.
\end{cor}

\begin{preuve}
In this context, all the jacobian blocks are empty, and the
correspondence blocks can always be reduced to a single line.
The \'etaleness of the projection~$\pi_2$ (resp.~$\pi_1$) at a
point~$(P,Q)\in \Cdeux_{f,g}$ becomes~$f'(P) \not= 0$ (resp.~$g'(Q) \not=0$).
The corollary follows.
\end{preuve}

\bigbreak

We will need at the end of this paper a characterization of the singular points contained in the
intersection of the curve~$C_n$ and the hypersurface~$P_1=P_n$ in~$\Cun^n$.
In the following proposition, we continue to make use of a local embedding of~$\Cun$
in~$\Aff^r$ and of the associated determinants~$\pi'_1(P,Q)$
and~$\pi'_2(P,Q)$ defined in (\ref{Deltaa}).

\begin{prop}\label{multiple}
Let~$(\Cun,\Cdeux)$ be a correspondence as in section~\ref{s_recursive_tower},
let~$(C_n)_{n\geq 1}$ be the associated recursive tower, and
let~$H_n$ denote the hypersurface of~$\Cun^n$ defined by~$P_1 = P_n$.

Consider~$(P_1,\ldots,P_n) \in C_n$ a smooth point. Then it is singular
in the intersection~$C_n \cap H_n$ if and only if
$$
(-1)^{r(n-1)} \prod_{i=1}^{n-1} \pi'_2(P_i,P_{i+1})
=
\prod_{i=1}^{n-1} \pi'_1(P_i,P_{i+1}),
\qquad \text{in~$\FF_q(P_1,\ldots,P_n)$,}
$$
where,~$\pi'_1$ and~$\pi'_2$ are the determinants defined by (\ref{Deltaa}).
\end{prop}

\begin{preuve}
We work in the affine neighborhood~$\Aff^{rn}$ of~$(P_1,\ldots,P_n)$.
If~$(P_1,\ldots,P_n) \in C_n\cap H_n$, that is if~$P_1 = P_n$,
then
$$
J_{C_n \cap H_n}(P_1,\ldots,P_n)
=
\begin{tikzpicture}[baseline=(M.center)]
\matrix (M) [matrix of math nodes,left delimiter=(,right delimiter=),
             nodes in empty cells,
             minimum width = 4em,
             column sep={1.5cm,between origins}]
{
&&\\
&J_{C_n}(P_1,\ldots,P_n)&\\
&&\\
I_r&&-I_r\\
};
\draw[rounded corners] (M-1-1.north west) rectangle (M-3-3.south east);
\draw[rounded corners] (M-4-1.north west) rectangle (M-4-1.south east);
\draw[rounded corners] (M-4-3.north west) rectangle (M-4-3.south east);
\draw[loosely dotted,thick] (M-4-1.east)--(M-4-3.west);
\end{tikzpicture}
$$
where~$I_r$ is the identity matrix of size~$r$.
Since~$(P_1,\ldots,P_n) \in C_n$ is assumed to be smooth, the jacobian
matrix~$J_{C_n}(P_1,\ldots,P_n)$ has rank equal to~$nr-1$.
As to the point~$(P_1,\ldots,P_n) \in C_n \cap H_n$, it is singular if and only
if the matrix~$J_{C_n \cap H_n}(P_1,\ldots,P_n)$ is still of rank~$rn-1$,
if and only if the $r$-th last
lines of~$J_{C_n \cap H_n}(P_1,\ldots,P_n)$
lie in the vector space generated by the lines of~$J_{C_n}(P_1,\ldots,P_n)$.

As in the proof of proposition~\ref{singular}, after cancellation
of redundant lines of the jacobian matrix~$J_{C_n}(P_1,\ldots,P_n)$,
we obtain a $(rn-1)\times rn$ matrix~$J'_{C_n}(P_1,\ldots,P_n)$ of maximal rank.
The singularity conditions then reduce to the vanishing of the $r$
determinants
$$
\begin{tikzpicture}[baseline=(M.center)]
\matrix (M) [matrix of math nodes,left delimiter=|,right delimiter=|,
             nodes in empty cells,
             minimum width = 4em,
             column sep={1.5cm,between origins}]
{
&&\\
&J'_{C_n}(P_1,\ldots,P_n)&\\
&&\\
1,0,\ldots,0&&-1,0,\ldots,0\\
};
\draw[rounded corners] (M-1-1.north west) rectangle (M-3-3.south east);
\draw[loosely dotted,thick] (M-4-1.east)--(M-4-3.west);
\end{tikzpicture}
=
\cdots
=
\begin{tikzpicture}[baseline=(M.center)]
\matrix (M) [matrix of math nodes,left delimiter=|,right delimiter=|,
             nodes in empty cells,
             minimum width = 4em,
             column sep={1.5cm,between origins}]
{
&&\\
&J'_{C_n}(P_1,\ldots,P_n)&\\
&&\\
0,\ldots,0,1&&0,\ldots,0,-1\\
};
\draw[rounded corners] (M-1-1.north west) rectangle (M-3-3.south east);
\draw[loosely dotted,thick] (M-4-1.east)--(M-4-3.west);
\end{tikzpicture}
=0.
$$
For each~$i$, $1\leq i \leq r$, expanding the determinant along the last line
leads, since~$P_1 = P_n$, to
$$
(-1)^{rn+i}
\delta_i(J_{\Cun}(P_1))
\prod_{i=1}^{n-1} \pi'_1(P_i,P_{i+1})
+
(-1)^{rn+r(n-1)+i}\times (-1) \delta_i(J_{\Cun}(P_1))
\prod_{i=1}^{n-1} \pi'_2(P_i,P_{i+1}) = 0.
$$
Since~$P_1\in \Cun$ is smooth, at least
one of the~$\delta_i(J_{\Cun}(P_1))$'s, $1\leq i \leq r$, is non-zero and the
result follows.
\end{preuve}

A cycle in $C_n$ is  a point $(P_1, \ldots, P_n) \in C_n \subset \Cun^n$ such that $P_1=P_n$.

\begin{cor} \label{bouclesing}
With the hypothesis of the preceding proposition, let~$(P_1,\ldots,P_{l},P_1)$ be a smooth cycle of length~$l\geq 1$ in~$C_{l+1}$.
Then there exists an iteration~$\rho$ of this cycle that becomes singular
in~$C_{\rho l + 1} \cap H_{\rho l + 1}$.
\end{cor}

\begin{proof}
Applying the preceding proposition, we know that the $\rho$-th iterate cycle
is singular if and only if
$$
(-1)^{r\rho l} \prod_{i=1}^{l} \pi'_2(P_i,P_{i+1})^\rho
=
\prod_{i=1}^{l} \pi'_1(P_i,P_{i+1})^\rho.
$$
Hence~$\rho$ equal to~$\#\FF_q(P_1,\ldots,P_l) -1 $ works.
\end{proof}

\section{Graphs and recursive towers}\label{s_graph}

\subsection{Some basic definitions and properties in graphs theory}
\label{s_basic_results_graphs}

We list all the definitions and properties we use from graph theory
in this section and the next one.
We  refer to Godsil \& Royle's GTM \cite{AGT} for proofs and more
details.

\smallbreak

{\bfseries Connectedness.~---}
An undirected graph is said {\em connected} if there exists a path from each vertex to every other vertex.
A {\em connected component}
of a graph is a maximal connected subgraph.

A directed graph is said to be {\em weakly connected} if it becomes connected by forgetting orientation.
A {\em weakly connected component} is a maximal weakly connected subgraph.
A directed graph
is said to be {\em strongly connected}
if there is a path from each vertex to every other vertex. A
{\em strongly connected component} of a graph is a maximal strongly
connected subgraph. Such a component is said to be {\em primitive}
if there is a path {\em of common length} between every couple of
vertices (see loc. cit. \S2.6).

\smallbreak

{\bfseries Regularness.~---} Let~$\Graphe$ be a directed graph and~$P$
one of its vertices. The {\em out degree}~$d^+(P)$
and {\em in degree}~$d^-(P)$ at~$P$ is the number of vertices~$Q$ such
that there exists a path from~$P$ to~$Q$ and from~$Q$ to~$P$.
A graph is said $d$-regular if and only if the in and out degrees at any
vertices is equal to~$d$.
For a finite {\em $d$-regular} directed graph, being weakly connected is
equivalent to being strongly connected
(see loc. cit. Lemma~2.6.1). As a consequence, every strongly connected
component of a finite {\em $d$-regular} directed graph is still $d$-regular
(indeed: any weakly connected component of such a graph must be $d$-regular
and weakly connected and then strongly connected).

\smallbreak

{\bfseries Adjacency matrix.~---} To each finite directed graph~$\Graphe$,
one can associate its {\em adjacency matrix}~$A$.
One can easily verify that
\begin{equation} \label{norme}
\#\{\text{cycles of length~$n$}\} = \tr(A^n)
\qquad\text{and}\qquad
\#\{\text{paths of length~$n$}\}
=
\vertiii{A^n}
\end{equation}
where~$\vertiii{(a_{i,j})_{i,j}}
= \sum_{i,j} |a_{i,j}|$.
Moreover, most of the previous properties
can be read off this matrix.
The graph is {\em $d$-regular} if and only if the sums of the
coefficients of every lines and of every columns equal~$d$.
It is {\em strongly connected} (resp. {\em primitive})
if and only if~$A$ is {\em irreducible} (resp. {\em primitive}).

\smallbreak

{\bfseries Spectral theory of non negative matrices.~---} The adjacency
matrix of a graph is of course a non negative matrix. Such matrices
have a well known spectral theory (see~\cite[Chapter~8]{MatrixAnalysis}).
One of the most important result in this area is the Perron-Frobenius theorem
(see loc. cit. Theorem~8.8.1). We will use it several times in the sequel.

\subsection{The geometric graph and its arithmetic and singular subgraphs}

To a recursive tower ${\mathcal T}(\Cun, \Cdeux)$ given by an irreducible correspondence~$\Cdeux$
on~$\Cun\times\Cun$, we associate in a very natural way an incidence ``geometric" infinite graph whose vertices are the geometric
points of~$\Cun$ and whose edges depend on~$\Cdeux$. Some of its ``arithmetic" finite subgraphs will play a crucial role till the end of this paper and in the proof of theorem \ref{betar}.

\begin{defi} \label{vocabulaire}
Let~$(\Cun,\Cdeux)$ be a correspondence as in section~\ref{s_recursive_tower}.
\begin{enumerate}
\item The {\bfseries geometric graph}
$\Graphe_{\infty}(\Cun, \Cdeux) = \Graphe_{\infty}$
is the graph whose vertices are the geometric points
of~$\Cun$, and for which there is an oriented edge
from $P \in \Cun(\overline{\FF_q})$ to $Q \in \Cun(\overline{\FF_q})$
if $(P, Q) \in \Cdeux(\overline{\FF_q})$.
\item An oriented edge $P \rightarrow Q$ of the graph is said to be \'etale by $\pi_1$ if the morphism $\pi_1$ is \'etale at $(P, Q)$. In the same way, the edge $P \rightarrow Q$ is said to be \'etale by $\pi_2$ if the morphism $\pi_2$ is \'etale at $(P, Q)$.
The {\bfseries singular part} of the graph~$\Graphe_\infty$,
denoted by~$\Gsing$, is the union of all weakly connected
components
containing at least one edge which is not \'etale by~$\pi_1$ or by~$\pi_2$.
\item For any subset~$S \subset \Cun(\overline{\FF_q})$, the
graph~$\Graphe_S$ is the subgraph of~$\Graphe_{\infty}$, whose vertices are
the points of $S$ and where there is an oriented edge from $P\in S$ to $Q \in S$ if there is one in $\Graphe_{\infty}$, that is if $(P, Q) \in \Cdeux(\overline{\FF_q})$.
\item In particular for~$S = \Cun(\FF_{q^r})$, $1 \leq r < +\infty$, we denote
by~$\Graphe_r$ the subgraph~$\Graphe_{\Cun(\FF_{q^r})}$ and we call it
the {\bfseries $r$-th arithmetic graph}.
\end{enumerate}
\end{defi}

This graph is a convenient way to ``see'' some of the most important features of a recursive tower:
 
$\bullet$ The
geometric points of~$C_n$ are in bijection with the paths of length $n-1$
of~$\Graphe_\infty$
(that is $n$ vertices and $n-1$ edges) while the arithmetic points defined
over~$\FF_{q^r}$ are in bijection with the paths of length $n-1$ of~$\Graphe_r$.

$\bullet$ The non \'etale points $(P, Q) \in \Cdeux$ can be read off the
in and out degrees of the graph~$\Graphe_\infty$.
Indeed, for every vertex~$P\in\Cun(\overline{\FF_q})$, the out-degree $d^+(P)$
(resp. in-degree $d^-(P)$)
at~$P$ of the graph~$\Graphe_\infty$ is equal to~$d$
{\it except} if there exists at least one point~$(P,Q)\in\Cdeux$
(resp.~$(Q,P)\in\Cdeux$) above~$P$
which is not \'etale by~$\pi_1$ (resp.~$\pi_2$), in which case this out (resp. in) degree is $<d$.

$\bullet$ The complementary part of the singular part
of $\Graphe_\infty$ is a $d$-regular graph in the graph theoretic sense, which means that at every vertex the
out and the in degrees are equal to~$d$.

\medbreak

One can even  be more precise.

\begin{prop}\label{lisse&totdec}
Let~$(\Cun,\Cdeux)$ be a correspondence as in section~\ref{s_recursive_tower}.
\begin{enumerate}
\item A path of length~$(n-1)$ in~$\Graphe_\infty$ corresponds to a singular
point of~$C_n$ if and only if there exist~$1\leq i\leq j<n$ such that
the edge~$P_{i}\to P_{i+1}$ is not \'etale by~$\pi_2$ and the
edge~$P_j\to P_{j+1}$ is not \'etale by~$\pi_1$. In particular,
this path is contained in the singular part~$\Gsing$ of~$\Graphe_\infty$.
\item Every path of length~$(n-1)$ outside the singular part
corresponds to a smooth point of~$C_n$.
\end{enumerate}
\end{prop}

\begin{preuve}
The first item is only a translation of proposition \ref{singular} whereas the second one follows trivially.
\end{preuve}

For instance, we represent in figure~\ref{graphe_degre_2} the second
arithmetic graph~$\Graphe_2$ for the very nice tame tower
of~\cite{GS_Degre2Moderee} recursively defined by $y^2=\frac{x^2+1}{2x}$
over~$\FF_5$. The singular part~$\Gsing$ is the subgraph whose vertices
are the points of~$\PP^1(\FF_5) = \{0, \pm 1,\pm 2,\infty\}$.

\begin{figure}
\centering
\begin{tikzpicture}[line width = 1pt]
\fill (0,0) circle (3pt) node[left] {$\alpha^7$} ;
\fill (2,1) circle (3pt) node[above left] {$\alpha^{19}$} ;
\fill (2,-1) circle (3pt) node[below left] {$\alpha^{17}$} ;
\fill (4,0) circle (3pt) node[left] {$\alpha$} ;
\fill (6,0) circle (3pt) node[right] {$\alpha^5$} ;
\fill (8,1) circle (3pt) node[above right] {$\alpha^{13}$} ;
\fill (8,-1) circle (3pt) node[below right] {$\alpha^{23}$} ;
\fill (10,0) circle (3pt) node[right] {$\alpha^{11}$} ;
\draw[decoration={markings, mark=at position 0.625 with {\arrow{>}}},
      postaction={decorate}]
     (-0.5,0) circle (0.5cm) ;
\draw[decoration={markings, mark=at position 0.5 with {\arrow{>}}},
      postaction={decorate}] (0,0) to [bend left] (2,1) ;
\draw[decoration={markings, mark=at position 0.5 with {\arrow{>}}},
      postaction={decorate}] (2,-1) to [bend left] (0,0) ;
\draw[decoration={markings, mark=at position 0.5 with {\arrow{>}}},
      postaction={decorate}] (2,-1) to [bend left] (2,1) ;
\draw[decoration={markings, mark=at position 0.5 with {\arrow{>}}},
      postaction={decorate}] (2,1) to [bend left] (4,0) ;
\draw[decoration={markings, mark=at position 0.5 with {\arrow{>}}},
      postaction={decorate}] (4,0) to [bend left] (2,-1) ;
\draw[decoration={markings, mark=at position 0.5 with {\arrow{>}}},
      postaction={decorate}] (4,0) to [bend left] (6,0) ;
\draw[decoration={markings, mark=at position 0.5 with {\arrow{>}}},
      postaction={decorate}] (6,0) to [bend left] (4,0) ;
\draw[decoration={markings, mark=at position 0.5 with {\arrow{>}}},
      postaction={decorate}] (6,0) to [bend left] (8,1) ;
\draw[decoration={markings, mark=at position 0.5 with {\arrow{>}}},
      postaction={decorate}] (8,-1) to [bend left] (6,0) ;
\draw[decoration={markings, mark=at position 0.5 with {\arrow{>}}},
      postaction={decorate}] (8,1) to [bend left] (10,0) ;
\draw[decoration={markings, mark=at position 0.5 with {\arrow{>}}},
      postaction={decorate}] (10,0) to [bend left] (8,-1) ;
\draw[decoration={markings, mark=at position 0.5 with {\arrow{>}}},
      postaction={decorate}] (8,1) to [bend left] (8,-1) ;
\draw[decoration={markings, mark=at position 0.5 with {\arrow{>}}},
      postaction={decorate}] (2,1) to [bend left] (8,1) ;
\draw[decoration={markings, mark=at position 0.5 with {\arrow{>}}},
      postaction={decorate}] (8,-1) to [bend left] (2,-1) ;
\draw[decoration={markings, mark=at position 0.2 with {\arrow{>}}},
      postaction={decorate}]
     (10.5,0) circle (0.5cm) ;
\end{tikzpicture}

\begin{tikzpicture}[line width = 1pt]
\fill (0,0) circle (3pt) node[left] {$1$} ;
\fill (2,0) circle (3pt) node[above left] {$-1$} ;
\fill (4,1) circle (3pt) node[above] {$2$} ;
\fill (4,-1) circle (3pt) node[below] {$-2$} ;
\fill (6,0) circle (3pt) node[above right] {$0$} ;
\fill (8,0) circle (3pt) node[right] {$\infty$} ;
\draw[decoration={markings, mark=at position 0.625 with {\arrow{>}}},
      postaction={decorate}]
     (-0.5,0) circle (0.5cm) ;
\draw[decoration={markings, mark=at position 0.5 with {\arrow{>}}},
      postaction={decorate}] (0,0) to (2,0) ;
\draw[decoration={markings, mark=at position 0.5 with {\arrow{>}}},
      postaction={decorate}] (2,0) to [bend left] (4,1) ;
\draw[decoration={markings, mark=at position 0.5 with {\arrow{>}}},
      postaction={decorate}] (2,0) to [bend right] (4,-1) ;
\draw[decoration={markings, mark=at position 0.5 with {\arrow{>}}},
      postaction={decorate}] (4,1) to [bend left] (6,0) ;
\draw[decoration={markings, mark=at position 0.5 with {\arrow{>}}},
      postaction={decorate}] (4,-1) to [bend right] (6,0) ;
\draw[decoration={markings, mark=at position 0.5 with {\arrow{>}}},
      postaction={decorate}] (6,0) to (8,0) ;
\draw[decoration={markings, mark=at position 0.2 with {\arrow{>}}},
      postaction={decorate}]
     (8.5,0) circle (0.5cm) ;
\end{tikzpicture}

\begin{tikzpicture}[line width = 1pt]
\fill (0,0) circle (3pt) node[left] {$\alpha^4$} ;
\fill (2,1) circle (3pt) node[above] {$\alpha^9$} ;
\fill (2,-1) circle (3pt) node[below] {$\alpha^{21}$} ;
\fill (4,0) circle (3pt) node[right] {$\alpha^{20}$} ;
\draw[decoration={markings, mark=at position 0.5 with {\arrow{>}}},
      postaction={decorate}] (0,0) to [bend left] (2,1) ;
\draw[decoration={markings, mark=at position 0.5 with {\arrow{>}}},
      postaction={decorate}] (4,0) to [bend right] (2,1) ;
\draw[decoration={markings, mark=at position 0.5 with {\arrow{>}}},
      postaction={decorate}] (0,0) to [bend right] (2,-1) ;
\draw[decoration={markings, mark=at position 0.5 with {\arrow{>}}},
      postaction={decorate}] (4,0) to [bend left] (2,-1) ;
\end{tikzpicture}
\begin{tikzpicture}[line width = 1pt]
\fill (0,0) circle (3pt) node[left] {$\alpha^8$} ;
\fill (2,1) circle (3pt) node[above] {$\alpha^3$} ;
\fill (2,-1) circle (3pt) node[below] {$\alpha^{15}$} ;
\fill (4,0) circle (3pt) node[right] {$\alpha^{16}$} ;
\draw[decoration={markings, mark=at position 0.5 with {\arrow{>}}},
      postaction={decorate}] (0,0) to [bend left] (2,1) ;
\draw[decoration={markings, mark=at position 0.5 with {\arrow{>}}},
      postaction={decorate}] (4,0) to [bend right] (2,1) ;
\draw[decoration={markings, mark=at position 0.5 with {\arrow{>}}},
      postaction={decorate}] (0,0) to [bend right] (2,-1) ;
\draw[decoration={markings, mark=at position 0.5 with {\arrow{>}}},
      postaction={decorate}] (4,0) to [bend left] (2,-1) ;
\fill (6,0.5) circle (3pt) node[above] {$\alpha^2$} ;
\fill (8,0.5) circle (3pt) node[above] {$\alpha^{10}$} ;
\fill (6,-0.5) circle (3pt) node[below] {$\alpha^{14}$} ;
\fill (8,-0.5) circle (3pt) node[below] {$\alpha^{22}$} ;
\end{tikzpicture}
\caption{The second arithmetic graph~$\Graphe_2\left(\PP^1, \frac{x^2+1}{2x} = y^2\right)$ over~$\FF_{25}$}\label{graphe_degre_2}
\end{figure}

\subsection{Finite complete sets and rational points}

We define as Beelen (see \cite{Beelen}) the notions of complete, backward complete and forward complete sets.

\begin{defi}
A subset~$S$ of~$\Cun(\overline{\FF_q})$ is said to be:
\begin{enumerate}
\item {\bfseries forward complete} if every point of~$S$ has
all its outgoing neighbors in~$\Graphe_\infty$ inside~$S$, that
is if $\pi_2(\pi_1^{-1}(S)) \subset S$;
\item  {\bfseries backward complete} if every point of~$S$ has
all its incoming neighbors in~$\Graphe_\infty$ inside~$S$, that
is if $\pi_1(\pi_2^{-1}(S)) \subset S$;
\item {\bfseries complete} if it is both backward and forward complete.
\end{enumerate}
\end{defi}

\begin{rk}
Being complete for a subset~$S\subset\Cun(\overline{\FF_q})$
does {\bfseries NOT} mean that
the graph~$\Graphe_S$ is complete in the usual sense of graph theory.
\end{rk}

If~$S$ is complete and if the subgraph $\Graphe_S$ is outside
the singular part~$\Graphe_{sing}$ then $\Graphe_S$ is $d$-regular in
standard graph theory. The following examples will illustrate this.

In the example of figure~\ref{graphe_degre_2}, the
sets~$\{\alpha,\alpha^{5},\alpha^{7},\alpha^{11},\alpha^{13},\alpha^{17},\alpha^{19},\alpha^{23}\}$ and~$\{0,\pm 1,\pm 2,\infty\}$ are complete
while the set~$\{\alpha^{4},\alpha^{9},\alpha^{20},\alpha^{21}\}$ is neither
forward nor backward complete. The set $\{2, 0, \infty\}$ is forward complete, but not backward complete.
 Moreover, the fact that, for instance, $\alpha^3$ possesses no outgoing edge means that there is no point in~$C_2(\FF_{25})$ above the
point~$\alpha^3 \in C_1(\FF_{25})$; this also means that
there is no points in~$C_3(\FF_{25})$
above the point~$(\alpha^{16}, \alpha^3) \in C_2(\FF_{25})$.
In other terms, this point is inert in the field
extension~$\FF_{25}(C_3)/\FF_{25}(C_2)$.

\begin{lem} \label{forward_complete_complete}
Let~$S$ be a finite and backward complete subset of~$\Cun(\overline{\FF_q})$ such that the graph~$\Graphe_S$ is outside the singular part. Then~$S$ is complete.
\end{lem}

\begin{preuve}
Since~$\Graphe_S$ is outside the singular
part and $S$ is backward complete, the in-degree at every vertex $P \in S$ is equal to~$d^{-}(P)=d$.
On the other side, the out-degree~$d^{+}(P)$ at each vertex~$P$ is less than~$d$.
Counting the edges, we get
$$
d\#S = \sum_{P\in S} d^-(P) = \sum_{P\in S} d^+(P) \leq d\#S,
$$
so that one must have~$d^+(P) = d$ for every~$P\in S$, which means that $S$ is also forward complete and thus complete.
\end{preuve}

\begin{prop}\label{partietotdec}
Let~$(\Cun,\Cdeux)$ be a correspondence as in section~\ref{s_recursive_tower}.
If there exists a finite complete set~$S \subset X({\mathbb F}_{q^r})$ such that
the graph~$\Graphe_S$ is outside the singular
part~$\Gsing$,
then
$$
\# \widetilde{C}_n(\FF_{q^r}) \geq \# S \times d^{n-1}.
$$
\end{prop}

\begin{preuve}
Since~$S$ is complete and since the graph~$\Graphe_S$ is assumed to be
outside the singular part, the graph~$\Graphe_S$
must be $d$-regular. Each path of length~$n$ in~$\Graphe_S$ gives rise to
exactly $d$ paths of length~$n+1$ by adding one of the $d$ 
outgoing neighbors of the ending vertex. All these paths correspond
to smooth points of~$C_n$ or~$C_{n+1}$ and we have just proved that above each
such point of~$C_n$, there are exactly $d$ points on~$C_{n+1}$. 
We easily conclude by induction since~$C_1$ counts at least $\#S$ points
defined over~$\FF_{q^r}$.
\end{preuve}

\subsection{Illustration with the BGS tower} \label{s_computation_genus}

In~$1985$, Bezerra, Garcia and Stichtenoth \cite{Ap3-3}
have introduced what we refer to as the BGS recursive
tower~${\mathcal T}(\Cun,\Cdeux)$ over~$\FF_q$,
defined by~$\Cun = \PP^1$ and by the separated variables
correspondence~$\Cdeux_{f,g}$ (notations of section~\ref{s_recursive_tower}) with
\begin{equation}\label{BGS-equation}
f(x) = \frac{x^q+x-1}{x}
\qquad\text{and}\qquad
g(y) = \frac{1-y}{y^q}.
\end{equation}
For each~$n\geq 1$, the curve~$C_n$ is embedded
in~$(\PP^1)^n = \prod_{i=1}^n\Proj(\FF_q[x_i,y_i])$ and is defined by the
ideal
$$
\left\langle
x_{i+1}^q\left(x_i^q+x_iy_i^{q-1}-y_i^q\right)
-
\left(y_{i+1}^q - x_{i+1}y_{i+1}^{q-1}\right)x_iy_i^{q-1},\, 1\leq i\leq n-1
\right\rangle.
$$
In this section, we would like to illustrate our approach about recursive
towers taking the BGS tower as example. We do not prove anything and refer
to the original Crelle's article \cite{Ap3-3} for the proofs.

\begin{figure}
\begin{tikzpicture}[line width = 1pt,scale=0.9,baseline=(current bounding box.west)]
\fill (0,0) circle (3pt) node[left] {$\alpha^{23}$} ;
\fill (2,1) circle (3pt) node[above left,xshift=-0.1cm] {$\alpha$} ;
\fill (2,-1) circle (3pt) node[below left] {$\alpha^{22}$} ;
\fill (2,2) circle (3pt) node[above left] {$\alpha^{18}$} ;
\fill (2,-2) circle (3pt) node[below left] {$\alpha^{2}$} ;
\fill (4,2) circle (3pt) node[above right,yshift=0.1cm] {$\alpha^{16}$} ;
\fill (4,-2) circle (3pt) node[above right,yshift=0.1cm] {$\alpha^3$} ;
\fill (6,2) circle (3pt) node[above right] {$\alpha^{9}$} ;
\fill (6,-2) circle (3pt) node[below right] {$\alpha^{14}$} ;
\fill (5,4) circle (3pt) node[above,yshift=0.2cm] {$\alpha^{25}$} ;
\fill (5,-4) circle (3pt) node[below,yshift=-0.2cm] {$\alpha^{17}$} ;
\fill (8,0) circle (3pt) node[right] {$\alpha^{6}$} ;
\draw[decoration={markings, mark=at position 0.625 with {\arrow{>}}},
      postaction={decorate}]
     (-0.5,0) circle (0.5cm) ;
\draw[decoration={markings, mark=at position 0.5 with {\arrow{>}}},
      postaction={decorate}] (0,0) to [bend left] (2,2) ;
\draw[decoration={markings, mark=at position 0.5 with {\arrow{>}}},
      postaction={decorate}] (0,0) to [bend right] (2,-1) ;

\draw[decoration={markings, mark=at position 0.5 with {\arrow{>}}},
      postaction={decorate}] (2,2) to [bend left] (4,2) ;
\draw[decoration={markings, mark=at position 0.5 with {\arrow{>}}},
      postaction={decorate}] (2,2) to [bend left] (5,4) ;
\draw[decoration={markings, mark=at position 0.6 with {\arrow{>}}},
      postaction={decorate}] (2,2) parabola bend (6,5.5) (8,0);

\draw[decoration={markings, mark=at position 0.5 with {\arrow{>}}},
      postaction={decorate}] (2,1) to [bend left] (2,-1) ;
\draw[decoration={markings, mark=at position 0.5 with {\arrow{>}}},
      postaction={decorate}] (2,1) to [bend left] (2,2) ;
\draw[decoration={markings, mark=at position 0.5 with {\arrow{>}}},
      postaction={decorate}] (2,1) to [bend right] (0,0) ;

\draw[decoration={markings, mark=at position 0.5 with {\arrow{>}}},
      postaction={decorate}] (4,2) to [bend left] (6,2) ;
\draw[decoration={markings, mark=at position 0.5 with {\arrow{>}}},
      postaction={decorate}] (4,2) to [bend left] (2,1) ;
\draw[decoration={markings, mark=at position 0.5 with {\arrow{>}}},
      postaction={decorate}] (4,2) to [bend right] (4,-2) ;

\draw[decoration={markings, mark=at position 0.625 with {\arrow{>}}},
      postaction={decorate}]
     (5,4.5) circle (0.5cm) ;
\draw[decoration={markings, mark=at position 0.5 with {\arrow{>}}},
      postaction={decorate}] (5,4) to [bend right] (4,2) ;
\draw[decoration={markings, mark=at position 0.5 with {\arrow{>}}},
      postaction={decorate}] (5,4) to [bend left] (8,0) ;

\draw[decoration={markings, mark=at position 0.5 with {\arrow{>}}},
      postaction={decorate}] (6,2) to [bend right] (5,4) ;
\draw[decoration={markings, mark=at position 0.5 with {\arrow{>}}},
      postaction={decorate}] (6,2) to [bend left] (4,2) ;
\draw[decoration={markings, mark=at position 0.5 with {\arrow{>}}},
      postaction={decorate}] (6,2) to [bend left] (8,0) ;

\draw[decoration={markings, mark=at position 0.6 with {\arrow{<}}},
      postaction={decorate}] (2,-2) parabola bend (6,-5.5) (8,0);

\draw[decoration={markings, mark=at position 0.5 with {\arrow{>}}},
      postaction={decorate}] (8,0) to [bend left] (5,-4) ;
\draw[decoration={markings, mark=at position 0.5 with {\arrow{>}}},
      postaction={decorate}] (8,0) to [bend left] (6,-2) ;

\draw[decoration={markings, mark=at position 0.5 with {\arrow{>}}},
      postaction={decorate}] (6,-2) to [bend right] (6,2) ;
\draw[decoration={markings, mark=at position 0.5 with {\arrow{>}}},
      postaction={decorate}] (6,-2) to [bend left] (4,-2) ;
\draw[decoration={markings, mark=at position 0.6 with {\arrow{>}}},
      postaction={decorate}] (6,-2) to [bend right] (2,1) ;

\draw[decoration={markings, mark=at position 0.625 with {\arrow{>}}},
      postaction={decorate}]
     (5,-4.5) circle (0.5cm) ;
\draw[decoration={markings, mark=at position 0.5 with {\arrow{>}}},
      postaction={decorate}] (5,-4) to [bend left] (2,-2) ;
\draw[decoration={markings, mark=at position 0.5 with {\arrow{>}}},
      postaction={decorate}] (5,-4) to [bend right] (6,-2) ;

\draw[decoration={markings, mark=at position 0.5 with {\arrow{>}}},
      postaction={decorate}] (4,-2) to [bend left] (2,-2) ;
\draw[decoration={markings, mark=at position 0.5 with {\arrow{>}}},
      postaction={decorate}] (4,-2) to [bend left] (6,-2) ;
\draw[decoration={markings, mark=at position 0.5 with {\arrow{>}}},
      postaction={decorate}] (4,-2) to [bend right] (5,-4) ;

\draw[decoration={markings, mark=at position 0.5 with {\arrow{>}}},
      postaction={decorate}] (2,-2) to [bend left] (0,0) ;
\draw[decoration={markings, mark=at position 0.5 with {\arrow{>}}},
      postaction={decorate}] (2,-2) to [bend left] (2,-1) ;
\draw[decoration={markings, mark=at position 0.5 with {\arrow{>}}},
      postaction={decorate}] (2,-2) to [bend right=60] (2,2) ;

\draw[decoration={markings, mark=at position 0.5 with {\arrow{>}}},
      postaction={decorate}] (2,-1) to [bend left] (2,1) ;
\draw[decoration={markings, mark=at position 0.5 with {\arrow{>}}},
      postaction={decorate}] (2,-1) to [bend right] (4,-2) ;
\draw[decoration={markings, mark=at position 0.6 with {\arrow{>}}},
      postaction={decorate}] (2,-1) to [bend right] (6,2) ;
\draw (-0.5,-4) node[anchor = west] {$\FF_{27} = \FF_3(\alpha)$} ;
\draw (-0.5,-4.7) node[anchor = west] {$\alpha^3 + 2\alpha + 1 = 0$} ;
\end{tikzpicture}
\begin{tikzpicture}[line width = 1pt,scale=0.9,baseline=(current bounding box.west)]
\fill (0,0) circle (3pt) node[left] {$1$} ;
\fill (2,2) circle (3pt) node[above] {$\theta_1$} ;
\fill (2,0) circle (3pt) node[below right] {$\theta_2$} ;
\fill (2,-2) circle (3pt) node[below] {$\theta_3$} ;
\fill (4,0) circle (3pt) node[below right] {$\infty$} ;
\fill (6,0) circle (3pt) node[below left] {$0$} ;
\draw[decoration={markings, mark=at position 0.5 with {\arrow{>}}},
      postaction={decorate}] (0,0) to [bend left] (2,2) ;
\draw[decoration={markings, mark=at position 0.5 with {\arrow{>}}},
      postaction={decorate}] (0,0) to [bend left=20] (2,0) ;
\draw[decoration={markings, mark=at position 0.5 with {\arrow{>}}},
      postaction={decorate}] (0,0) to [bend left=20] (2,-2) ;
\draw[decoration={markings, mark=at position 0.5 with {\arrow{>}}},
      postaction={decorate}] (2,2) to [bend left=20] (0,0) ;
\draw[decoration={markings, mark=at position 0.5 with {\arrow{>}}},
      postaction={decorate}] (2,0) to [bend left=20] (0,0) ;
\draw[decoration={markings, mark=at position 0.5 with {\arrow{>}}},
      postaction={decorate}] (2,-2) to [bend left] (0,0) ;

\draw[decoration={markings, mark=at position 0.5 with {\arrow{>}}},
      postaction={decorate}] (2,2) to [bend left] (4,0) ;
\draw[decoration={markings, mark=at position 0.5 with {\arrow{>}}},
      postaction={decorate}] (2,0) to (4,0) ;
\draw[decoration={markings, mark=at position 0.5 with {\arrow{>}}},
      postaction={decorate}] (2,-2) to [bend right] (4,0) ;

\draw[decoration={markings, mark=at position 0.5 with {\arrow{>}}},
      postaction={decorate}] (4,0) to (6,0) ;

\draw[decoration={markings, mark=at position 0.625 with {\arrow{>}}},
      postaction={decorate}]
     (6.5,0) circle (0.5cm) ;
\draw (5.5,-2) node {$\theta_i^3 + \theta_i - 1 = 0$} ;
\end{tikzpicture}
\caption{The two interesting components of~$\Graphe_6$ for~$q=3$}\label{figure_Ap3}
\end{figure}

\medbreak

{\bfseries\noindent The totally split points~---~} In
figure~\ref{figure_Ap3}, we represent two complete sets for $q=3$. The left hand one counts
$q(q+1)$ vertices, all contained in~$\FF_{q^3}$. For~$q=5,7...$, one  easily
sees evidence of the existence of a complete set of size~$q(q+1)$ outside the
singular part. By proposition~\ref{partietotdec}, if this is true, one should have
$$
\# \widetilde{C}_n(\FF_{q^3}) \geq q^{n}(q+1).
$$
Of course, this is not a proof of this fact, but only a convenient way
to see it. This is proved in Crelle \cite[Proposition~3.1]{Ap3-3}.

\medbreak

{\bfseries\noindent The singular points~---~} The right hand side complete set
of figure~\ref{figure_Ap3} has points~$0,1,\infty$ as vertices.
These are exactly the ramified points of~$f$ or~$g$: the ramified points
of~$f$ (resp.~$g$) are~$1$ and~$\infty$
(resp.~$\infty$ and~$0$). The set~$\{0,1,\infty\}$ is not complete but
one can easily prove that it suffices to add the
set~${\mathcal R}$ of roots
of~$x^q+x-1$ to complete the set. The
subgraph~$\Graphe_{\{0,1,\infty\} \cup {\mathcal R}}$ is nothing else that
the singular part~$\Gsing$.

The fact that every curve~$C_n$ is irreducible can be read of
this component. Indeed, there is only one loop starting from the vertex~$0$.
This means that above the point~$0 \in C_1$, there is only one point,
i.e.~$(0,\ldots,0) \in C_n$.
Since~$0$ is not a ramified point of~$f$, this point is
smooth in~$C_n$ and then must be totally ramified over~$0 \in C_1$.
Necessarily~$C_n$ is irreducible.

\begin{rk}
Using this argument, this is a general fact that if there exists in $\Graphe_{sing}$ only one loop outgoing from a point, \'etale by $\pi_2$, then the tower is irreducible. This is a common feature of many towers of the literature.
\end{rk}

 It can also be easily seen that
for~$\alpha_0,\ldots,\alpha_r$ in ${\mathcal R}$, the
points
$$
(1,\alpha_1,1,\alpha_2,\ldots,1,\alpha_r) \in C_{2r}
\qquad\text{and}\qquad
(\alpha_0,1,\alpha_1,1,\alpha_2,\ldots,1,\alpha_r) \in C_{2r+1}
$$
are smooth. Actually to be singular on~$C_n$ or~$C_n^\sharp$, a point must start
by~$1$ or~$\infty$ (a ramified point of~$f$)
or by~$\alpha \in {\mathcal R}$ (an incoming neighbor of a ramified point of~$f$)
and must end by~$0$
or~$\infty$ (a ramified point of~$g$). Thus there are two types of
singular points depending on the ending point:
$$
{\renewcommand{\arraystretch}{1.5}
\begin{array}{r|l|l}
\text{\bf Type} & \text{\bf Corresponding points on~$C_n$} & \text{\bf Range of~$r$} \\
\hline
T_\infty
   & (1,\alpha_1,1,\alpha_2,\ldots,1,\alpha_{r},\infty)
   & \text{$n$ odd and~$r = \frac{n-1}{2}$}\\
   & (\alpha_0,1,\alpha_1,1,\alpha_2,\ldots,1,\alpha_{r},\infty)
   & \text{$n$ even and~$r = \frac{n-2}{2}$} \\
\hline
T_0
   & (1,\alpha_1,1,\alpha_2,\ldots,1,\alpha_r,\infty,0,\ldots,0)
   & 0\leq r\leq \left\lfloor\frac{n-2}{2}\right\rfloor\\
   & (\alpha_0,1,\alpha_1,1,\alpha_2,\ldots,1,\alpha_r,\infty,0,\ldots,0)
   & 0\leq r\leq \left\lfloor\frac{n-3}{2}\right\rfloor
\end{array}}
$$
In this table, the integer~$r$ is the number of instances of
couples~$(1,\alpha)$ for~$\alpha\in{\mathcal R}$ in the considered
point of~$C_n$. If~$r=0$, there is no
such couples; the two $r=0$ cases in type~$T_0$
are points~$(\infty,0,\ldots,0)$ and~$(\alpha_0,\infty,0,\ldots,0)$.

\medbreak

Finding an exact formula or even an upper bound for the genus sequence
in this tower is
a pretty hard and technical problem;
at least three articles deal with this specific problem
in the literature \cite{Ap3-3,Ap3-1,Ap3-2}.
We have tried to compute the genus sequence in the
spirit of section~\ref{s_genre_geometrique_arithmetique},
using standard techniques of curves desingularization, such as Newton
polygons and local integral closure computations.
Unfortunately, having made lots of preliminary calculations, we do not
think that our approach can give rise to a simpler proof for the genus
formula. Let us just reformulate results of
propositions~$2.7$
and~$2.8$ in Crelle's \cite{Ap3-3} in terms of desingularization.
The points of type~$T_\infty$ give rise to a unique point after
desingularization; this is not difficult to prove even with
the measure of singularity.
Points of type~$T_0$ are much more harder to deal with.
Points~$(iv)$ of locally cited propositions say that:
if $s$ is odd (respectively even), and~$n \geq 2s-1$
(respectively~$n \geq 2s-2$), then
a point of type~$T_0$ gives rise to $q^{\frac{(s-1)}{2}}$
(respectively~$q^{\frac{(s-2)}{2}}$) points; in between, this is more tricky.
For our purpose, the key point is that:

\begin{prop} \label{prop_genre_Ap3} 
The number of geometric points of $\widetilde{C}_n$ coming from the
desingularization of $C_n$ is~$O(\sqrt{q^n})$.
\end{prop}

The lower bound for the number of rational points over~$\FF_{q^3}$ and the upper
bound for the genus sequence leads to the inequalities
$$
A(q^3) \geq \lambda_3(\Cun, \Cdeux) \geq \frac{2(q^2-1)}{q+2}.
$$

\section{Application to the asymptotic behaviour of recursive towers} \label{s_asymptotic}

Collecting all results of the previous sections, we prove here our main theorem~\ref{betar}, which states that (under some assumption) at most one $\beta_r$ is non zero for a recursive tower. We begin by proving proposition~\ref{nombrecycles}, which mixes combinatoric and intersection theory.  In conjunction with a diophantine lemma~\ref{dioapprox} and some spectral considerations on the graph, we deduce theorem~\ref{uniquecomposympa}, that under the assumptions of section~\ref{s_recursive_tower}, there exists at most one finite strongly connected component in the geometric graph~$\Graphe_{\infty}$.  This theorem~\ref{uniquecomposympa} and considerations on non-negative matrices and Perron-Frobenius theory then lead us to proposition~\ref{cle}, a very precise form of the connection between adjacency matrices, subgraphs and number of rational points on curves. After deducing proposition~\ref{d-reg_ou_plein_de_pts_dans_desingularisation}, a second necessary condition for a tower to be good, we prove our main theorem~\ref{betar}. As an example, we compute using our results some invariants defined by Tsfasmann and Vl{\u{a}}du{\c{t}} \cite{TV02} for the BGS tower already studied in section \ref{s_computation_genus}. We deduce corollary~\ref{basechange}, which gives the analogous statement for pull-back of recursive towers, a family of towers recently studied by some authors.

\subsection{Number of cycles}\label{NombreCycles}

For the statement of the next proposition, we recall that the  class in the Neron-Severi group $\NS(\Cun \times \Cun)_{\mathbb R}$ of a correspondence $C$ on $\Cun$ is a triple $(d_1, d_2, \sigma) \in {\mathbb Z}\times {\mathbb Z} \times \End\left(T_\ell(\Jac(X))\right)$ where $\Jac(X)$ is the Jacobian variety of $X$, and $T_\ell(\Jac(X))$ is its Tate module for some prime number $\ell$ prime to $q$. For instance, the class of the diagonal $\Delta$ is $(1, 1, \Id)$.
Then, the intersection number $C \cdot C'$ is given by 
\begin{equation} \label{produitdintersection}
C\cdot C'  = (d_1, d_2, \sigma)\cdot (d_1', d_2', \sigma') = d_1d_2'+d_1'd_2-\tr(\sigma \sigma').
\end{equation}
Moreover, Castelnuevo identity states that the bilinear form $\tr(\sigma \sigma')$ is negative definite (\cite{Zariski} chapter VII, appendix of Mumford p. 153). It is worth noticing that what we called up to now the
 type $(d_1, d_2)$ of a divisor $C$ in 
 $X\times X$ 
 is actually the``trivial" part  of its complete  class $(d_1, d_2, \sigma)$.

Note that Weil \cite{Weil} defined a product on the set of correspondences as follows. Let $C$, $C'$ be two correspondences on $\Cun$. Then the composition $C\circ C'$ is the correspondence on $\Cun$ given by:
$$
\begin{array}{ccccc}
\Cun &\longrightarrow & \Div(\Cun) & \longrightarrow &\Div(\Cun)\\
P
&\longmapsto
& \sum_{Q \in \Cun \mid (P, Q) \in C'} Q
& \longmapsto
&\sum_{\genfrac{}{}{0pt}{}{Q \in \Cun \mid (P, Q) \in C'}{R \in \Cun \mid (Q, R) \in C}} R.
\end{array}
$$
It then holds that, if $C, C'$ have classes $(d_1, d_2, \sigma)$ and $(d'_1, d'_2, \sigma')$ in $\NS(X \times X)_{\mathbb R}$, then $C\circ C'$ has class $(d_1d'_1, d_2d'_2, \sigma \sigma')$. Of course, the class of $C+C'$ is $(d_1+d'_1, d_2+d'_2, \sigma_1+\sigma_2)$. Since the rational Neron-Severi group $\NS(\Cun \times \Cun)_{\mathbb Q}$ is finite dimensional, say of dimension $\rho$, the classes $\Delta, C, C\circ C, \cdots, C^{\rho}$ are ${\mathbb Q}$ linearly dependant. This gives a ${\mathbb Q}$-linear relation between $\Id, \sigma, \sigma^2, \cdots, \sigma^{\rho}$, which implies that the eigenvalues of $\sigma$ are in fact algebraic numbers. 

\medskip

The nice feature in the following statement is that the left hand side of
formula~$(\ref{comb-geom})$ is combinatorial in nature, whereas the right hand side is geometric. In this statement, we denote by $\Sp (u)$ the spectrum of an operator $u$.


\begin{prop} \label{nombrecycles}
Let~$(\Cun,\Cdeux)$ be a correspondence as in section~\ref{s_recursive_tower},
let~$(C_n)_{n\geq 1}$ be the associated recursive tower whose
curves~$C_n$ are assumed to be irreducible.
We denote by $\pi_{1, n+1}$ the projection map $X^{n+1} \rightarrow X\times X$ which sends $(P_1, \dots, P_{n+1})$ to $(P_1, P_{n+1})$ and by $\Delta$ the diagonal of $\Cun \times \Cun$.

\begin{enumerate}
\item\label{scheme_cap} The scheme-theoretic intersection~$C_{n+1} \cap \pi_{1, n+1}^*(\Delta)$
in~$X^{n+1}$ is zero-dimensional and of degree equal to
$$
C_{n+1} \cdot \pi_{1, n+1}^*(\Delta)
=
2d^n - \sum_{\mu \in \Sp(\sigma)} \mu^n,
$$
where~$(d, d, \sigma)$ is the  class of~$\Cdeux$ in~$\NS(\Cun\times\Cun)_\RR$.
\item\label{one-one} There is a one-one correspondence between the geometric points
of~$C_{n+1} \cap \pi_{1, n+1}^*(\Delta)$ and the cycles of
length $n$ in~$\Graphe_{\infty}$, whose number~$c_n$ is thus finite.
\item\label{c_n} Let~$r\geq 1$ be such that the graph~$\Graphe_r$ contains the cycles of
length~$n$ and let~$A_r$ be the adjacency matrix of~$\Graphe_r$.
Then
\begin{equation} \label{comb-geom}
c_n = \sum_{\lambda \in \Sp(A_r)} \lambda^n \leq 2d^n - \sum_{\mu \in \Sp(\sigma)} \mu^n
\end{equation}
and the last inequality is strict if the scheme $C_{n+1} \cap \pi_{1,n+1}^*(\Delta)$
contains a point with multiplicity at least $2$.
\end{enumerate}
\end{prop}

\begin{preuve}
We begin by proving that the irreducible curve $C_{n+1}$ is not contained in the hypersurface $\pi_{1,n+1}^*(\Delta)$.
From the hypothesis on $\Cdeux$, the first projection $\pi_1 : \Cdeux \rightarrow \Cun$ is
a finite morphism of degree $d$, \'etale except at a finite number of geometric points $(P, Q) \in \Cdeux$. Choose a geometric point $P \in \Cun$, such that $\pi_1$ is \'etale at any point $(P, Q_i) \in \Cdeux, 1 \leq i \leq d$
lying above $P$. 
Choose a geometric point $(P_1, \dots, P_{n-2}, P) \in C_{n-1}$ whose last coordinate is $P_{n-1}=P$. There are $d$ distinct geometric points $(P_1, \dots, P_{n-2}, P, Q_i) \in C_n$, for $1 \leq i \leq d$ lying above $(P_1, \dots, P_{n-2}, P) \in C_{n-1}$. Suppose now by contradiction that 
$C_n \subset  \pi_{1,n+1}^*(\Delta)$. This means that for any $1 \leq i \leq d$, we have $Q_i = P_1$, a contradiction since $d \geq 2$.
It follows that the intersection $C_{n+1} \cap \pi_{1,n+1}^*(\Delta)$ in $X^{n+1}$
is a zero dimensional subvariety.
By the projection formula, one
has~$
C_{n+1} \cdot \pi_{1,n+1}^*(\Delta) 
=
\left(\pi_{1, n+1}\right)_*(C_{n+1}) \cdot \Delta
$.
By definition, $\left(\pi_{1, n+1}\right)_*(C_{n+1})$ is nothing else than $\Cdeux \circ \Cdeux \circ \cdots \circ \Cdeux$ ($n$ times), hence its  class is $(d^n, d^n, \sigma^n)$. Now, the class of the diagonal equals $(1, 1, \Id)$, hence by (\ref{produitdintersection})
$$
C_{n+1} \cdot \pi_{1,n+1}^*(\Delta)
=
\left(\pi_{1, n+1}\right)_*(C_{n+1}) \cdot \Delta
=
d^n\times 1 + d^n \times 1 - \tr(\sigma^n),
$$
which proves~$\ref{scheme_cap}$.

Now, a geometric point $(P_1, \dots, P_{n+1}) \in \Cun^{n+1}$ corresponds to a cycle of length $n$ in $\Graphe_{\infty}$ if and only if $(P_1, \dots, P_{n+1}) \in C_{n+1}$ and $(P_1, \dots, P_{n+1}) \in \pi_{1, n+1}^*(\Delta)$, which means that cycles of length $n$ correspond to points in the zero-dimensional
intersection $C_{n+1} \cap \pi_{1,n+1}^*(\Delta)$ in $X^{n+1}$. This proves that~$c_n$ is finite
and that~$c_n \leq C_{n+1} \cdot \pi_{1,n+1}^*(\Delta)$.

Last, the equality~$c_n = \sum_{\lambda \in \Sp(A_r)} \lambda^n$ holds since the
number of cycles of length $n$ is by~(\ref{norme}) the trace of the $n$-th
power of the adjacency matrix $A_r$ of $\Graphe_r$.
\end{preuve}

\begin{rk}
In case $\Cun = {\mathbb P}^1$, it is also possible to give a proof of this proposition using resultants.
\end{rk}

This proposition is fruitful in conjunction with the following lemma:

\begin{lem}[Diophantine approximation]\label{dioapprox}
Let~$\lambda_1, \ldots,\lambda_k \in\CC^*$. Then there exists an integer $N \in {\mathbb N}^*$, such that
$\Re(\lambda_j^N) > 0$ for each~$1\leq j\leq k$.
\end{lem}

\begin{preuve} Let $\mu_j = \frac{\lambda_j}{\vert \lambda_j\vert}$ for $1 \leq j \leq k$. Then $\mu_j = \exp (2\imath \pi \theta_j)$ for some real number $\theta_j \in \RR$. It follows from Hardy and Wright \cite[theorem 201]{Hardy-Wright} that for any $\varepsilon >0$, there exists some $N \in {\mathbb N}^*$ such that $d(N\theta_j, \ZZ) < \varepsilon$ for all $1 \leq j \leq k$. By continuity of the exponential map, we can also choose $N$ such that $\vert \mu_j^N -1\vert < \varepsilon$ for any $1 \leq j \leq k$. Now we choose $\varepsilon = 1$. There exists some $N \in {\mathbb N}$, such that $\vert \lambda_j^N - \vert \lambda_j^N\vert \vert < \vert \lambda_j^N\vert$ for any $j$, which implies that $\Re e(\lambda_j^N) > 0$.
\end{preuve}

\subsection{Finite strongly connected regular components}\label{s_scc}

We should focus on {\em strongly connected components} of the
graph~$\Graphe_\infty$ and especially the finite ones.

\begin{prop} \label{primitive}
Let~$(\Cun,\Cdeux)$ be a correspondence as in section~\ref{s_recursive_tower}.
Then every finite $d$-regular strongly connected subgraph~$\Graphe$
of the graph~$\Graphe_\infty(\Cun, \Cdeux)$ is primitive.
\end{prop}

\begin{preuve}
Let~$A$ be the adjacency matrix of the subgraph~$\Graphe$.
Since~$\Graphe$ is supposed
to be strongly connected, the matrix~$A$ is irreducible.
Since~$\Graphe$ is $d$-regular, the vector~$(1, \dots, 1)$
is an eigenvector of~$A$ for the eigenvalue~$d$. By Perron-Frobenius
theorem, this eigenvalue is simple and is nothing else than the spectral
radius of~$A$. Moreover there exists a primitive root of unity~$\zeta_m$
such that the eigenvalues of absolute value~$d$
are the~$d\zeta_m^i$ for~$0\leq i \leq m-1$,
and all these eigenvalues are also simple.

Relating to the trace of the matrix~$A^{mn}$ for~$n\geq 1$, this implies that
$$
\tr(A^{mn})
=
m d^{mn} + \sum_{\genfrac{}{}{0pt}{2}{\lambda \in \Sp(A)}{|\lambda|<d}} \lambda^{mn}
$$
But this trace is also the number of cycles of length~$mn$ in~$\Graphe$.
By proposition~\ref{nombrecycles}, we thus have
$$
\forall n \geq 1,
\qquad
(m-2) d^{mn}
+ \sum_{\genfrac{}{}{0pt}{2}{\lambda \in \Sp(A)}{|\lambda|<d}} \lambda^{mn}
+ \sum_{\mu \in \Sp(\sigma)} \mu^{mn}
\leq 0,
$$
where~$(d,d,\sigma)$ is the class of~$\Cdeux$ in~$\NS(\Cun\times\Cun)_\RR$.
Note that in the left sum, any~$\lambda \in \CC\setminus\RR$ appears together
with its conjugate~$\bar{\lambda}$. Then this sum is, in fact, a sum
of real parts of powers of complex numbers.
By lemma~\ref{dioapprox}, we deduce that~$m\leq 2$. Moreover, if~$m=2$,
 all the eigenvalues~$\lambda$'s in the left sum must be equal to zero
using lemma~\ref{dioapprox} another time.
Then the number of cycles of length~$2n$ in~$\Graphe$, counted without
multiplicities, is exactly~$2d^{2n}$. But by proposition~\ref{nombrecycles},
it is also equal to~$2d^{2n}$ counted with multiplicities. Then,
for all~$n\geq 1$, every
cycle of length~$2n$ must be simple.
Thanks to corollary~\ref{bouclesing}, we know that this is
impossible. Hence~$m=1$; this characterizes the fact that the matrix~$A$
or the graph~$\Graphe$ are primitive.
\end{preuve}

\begin{theo} \label{uniquecomposympa}
Let~$(\Cun,\Cdeux)$ be a correspondence as in section~\ref{s_recursive_tower} such that
the curves~$C_n$ of the associated tower are all irreducible.
Then the graph~$\Graphe_\infty(\Cun,\Cdeux)$ has at most one finite $d$-regular strongly connected
component.
\end{theo}

\begin{rk}
This theorem contains as a particular case Beelen's theorem 5.5 \cite{Beelen} in the case of towers he called {\em of type A} on $\Cun = {\mathbb P}^1$.
\end{rk}

\begin{preuve}
Suppose that there exists at least one such component. Let~$\Graphe_1, \dots, \Graphe_k$, some finite $d$-regular strongly connected components of~$\Graphe_{\infty}$
and let~$A_i$, $1\leq i\leq k$, be their adjacency matrices.
We denote by~$\Sp(A_i)$ the spectrum of each~$A_i$.
As noticed in the preceding proof, each matrix~$A_i$ has  spectral
radius~$d$, and $d$ is a simple eigenvalue. Hence, for any $n \geq 1$
$$
\tr(A_1^n) + \dots + \tr(A_k^n)
=
k d^n + \sum_{\lambda \in \bigcup_{i=1}^k \Sp(A_i) \setminus \{d\}}\lambda^n.
$$

But this sum of traces is also the number of cycles of length~$n$ in the union of the $\Graphe_i$'s for $1 \leq i \leq k$, which is of course less than the number of cycles of length $n$ in the arithmetic graph $\Graphe_r$ for $r$ large enough. Now, we have assumed that there exists at least one finite $d$-regular strongly connected component, which contains of course at least one cycle of some length. Taken sufficiently often, this cycle has multiplicity at least $2$ by corollary \ref{bouclesing}, that is there is at least one cycle, of some length $m \in {\mathbb N}^*$, having multiplicity at least $2$.
%
Thus, by item~$\ref{c_n}$ of proposition~\ref{nombrecycles}, we have
for any $n \geq 1$,
\begin{equation}\label{moinsun}
k d^{mn} + \sum_{\lambda \in \bigcup_{i=1}^k \Sp(A_i) \setminus \{d\}}\lambda^{mn}
+\sum_{\mu \in \Sp(\sigma)} \mu^{mn}\leq 2d^{mn}-1
\end{equation}
that is
$$\sum_{\lambda \in \Sp(\sigma) \bigcup (\Sp(A_1)\setminus \{d\})\bigcup \dots  \bigcup (\Sp(A_k) \setminus \{d\})}(\lambda^m)^n \leq(2-k)d^{mn} -1.$$
Due to lemma~\ref{dioapprox}, there exists some $N \in {\mathbb N}^*$ such that
$$0 \leq (2-k)d^{mN}-1,$$
which implies that $k \leq 1$.
\end{preuve}

\begin{rk}
It is worth noticing that the uniqueness comes from the ``$-1$" at the end of formula (\ref{moinsun}), which itself comes from the whole section \ref{SectionSingular} through its final statement corollary \ref{bouclesing}. Without this term, the conclusion would be that there were at most {\em two} finite $d$-regular strongly connected components. Consequences on the rest of this paper would be that there were at most two (resp. $2d$) non-zero $\beta_r$'s in theorem \ref{betar} (resp. in corollary \ref{basechange}). 
\end{rk}

The following proposition describes accurately the connection between three different worlds.

\begin{prop}\label{cle}
Let~$(\Cun,\Cdeux)$ be a correspondence as in section~\ref{s_recursive_tower}, $S \subset \Cun(\overline{\FF_q})$ be a finite subset, $\Graphe_S$
be the corresponding subgraph of~$\Graphe_\infty$
and~$A_S$ be the adjacency matrix of~$\Graphe_S$.
Then, the following assertions are equivalent:
\begin{enumerate}
\item the spectral radius~$\rho(A_S)$ of~$A_S$ equals~$d$;\label{rhoAeqd}
\item there exists a unique~$\Sigma \subset S$ such that~$\Graphe_\Sigma$ is
$d$-regular and strongly connected;\label{Sigmadreg}
\item there exists~$c > 0$, such that: \label{pleinpoints}
$$
\sharp\{(P_1, \ldots, P_{n+1}) \in C_{n+1}(\overline{{\mathbb F}_q}) \mid \;
P_i \in S, \, \forall i\} = c\times d^n + o(d^n).
$$
\end{enumerate}
Moreover, if these assertions are true, then the constant $c$
in~$\ref{pleinpoints}$ equals $\sharp \Sigma$ in~$\ref{Sigmadreg}$;
otherwise,
$\sharp\{(P_1, \ldots, P_{n+1}) \in C_{n+1}(\overline{{\mathbb F}_q}) \mid P_i \in S,\, \forall i\}
= o(d^n)$.
\end{prop}

\begin{preuve}
$\ref{rhoAeqd}\Rightarrow\ref{Sigmadreg}$
Another consequence of 
Perron-Frobenius theorem is 
theorem~8.3.1 in Horn \& Johnson's book on matrix
analysis \cite{MatrixAnalysis} which states 
that the spectral radius~$d$ of~$A_S$ is an eigenvalue and that $d$
is associated to a non-negative eigenvector 
$$
u = (u_P)_{P \in  S}, \hbox{ such that $u_P \geq 0$ for any } P  \in  S.
$$
Of course, $u_P$ may be zero for some $P \in S$.
Let
$\Sigma \subset S$
be the set of~$P \in S$ such that $u_P \neq 0$.

Let~$A_S = (a_{P,Q})_{P,Q\in S}$ and~$A_\Sigma = (a_{P,Q})_{P,Q\in\Sigma}$. Then
$$
\forall P\in\Sigma,
\qquad
\sum_{Q\in S} a_{P,Q}u_Q = \sum_{Q\in\Sigma} a_{P,Q}u_Q = d u_P,
$$
which means that the positive vector~$(u_P)_{P\in\Sigma}$ is an eigenvector
of~$A_\Sigma$ for the eigenvalue~$d$. The matrix~$A_\Sigma$ is nothing else than
the adjacency matrix of the subgraph~$\Graphe_\Sigma$.

Now, we prove that $\Graphe_{\Sigma}$ is $d$-regular, outside the singular part.
By summation
$$
\sum_{Q \in \Sigma}\left(\sum_{P \in \Sigma}a_{P, Q}\right) u_Q = \sum_{Q \in \Sigma}d u_Q.
$$
But:
\begin{itemize}
\item each $u_Q$ is $>0$ for $Q \in \Sigma$,
\item each $(\sum_{P \in \Sigma}a_{P, Q})$ satisfies $\sum_{P \in \Sigma}a_{P, Q} \leq d$.
\end{itemize}
Hence
$$
\forall Q \in \Sigma, \qquad \sum_{P \in \Sigma}a_{P, Q} = d.
$$
Let $Q \in \Sigma$; each term in the $Q$'th column in the adjacency matrix $A_{\Sigma}$ contains exactly $d$ coefficients $1$. This means first that $\pi_1$ is \'etale at any edge exiting from $Q$, second that~$\Sigma$ is forward complete. Using a similar argument with the lines of $A_{\Sigma}$, we also prove that $\pi_2$ is \'etale at any edge entering at $Q$, and that $\Sigma$ is also backward complete, hence complete. In conclusion, the graph~$\Graphe_\Sigma$
is $d$-regular and any of its strongly connected components works. Now, such a strongly connected $d$-regular component is unique by theorem~\ref{uniquecomposympa}.

$\ref{Sigmadreg}\Rightarrow\ref{pleinpoints}$ As already noted, there is
a one-to-one correspondence between the set of
points of~$C_{n+1}$ with coordinates in~$S$ and the set paths
of length~$n$ in~$\Graphe_S$. Therefore (see section~\ref{s_basic_results_graphs}):
$$
\sharp\{(P_1, \cdots, P_{n+1}) \in C_{n+1}(\overline{{\mathbb F}_q}) \mid \;
P_i \in S, \, \forall i\} = \vertiii{A_S^n} 
$$
and we are reduced to compute an equivalent of this norm.

To this end, let~$\Graphe_{S_1}, \ldots, \Graphe_{S_r}$ be
the weakly connected components of~$\Graphe_S$.
The graph~$\Graphe_\Sigma$ is one of these components, say~$\Graphe_{S_1}$.
Let~$A_{S_1},\ldots,A_{S_r}$ be the corresponding adjacency matrices.
Then~$A_S$ is the block diagonal matrix whose blocks are
the~$A_{S_i}$'s and the norm satisfies
$$
\vertiii{A_S^n} = \sum_{i=1}^r \vertiii{A_{S_i}^n}.
$$
The first norm~$\vertiii{A_{S_1}^n}$ equals the number of paths of length $n$
in~$\Graphe_{S_1} = \Graphe_{\Sigma}$, that is equals $\sharp\Sigma \times d^n$
by $d$-regularity. 
For~$i\geq 2$, the spectral radius must satisfy~$\rho(A_{S_i}) < d$,
otherwise, by
$\ref{rhoAeqd}\Rightarrow\ref{Sigmadreg}$,
$\Graphe_{S_i}$ would contain a $d$-regular strongly
connected component. This would contradict the uniqueness stated
in theorem~\ref{uniquecomposympa}.
Gelfand spectral radius theorem states
that~$\rho(A_{S_i}) = \lim_{n \to \infty} \root{n}\of {\vertiii{A_{S_i}^n}}$.
Therefore,~$\rho(A_{S_i}) < d$ implies that~$\vertiii{A_{S_i}^n} = o(d^n)$.
We deduce that~$\vertiii{A_S^n} = \sum_{i=1}^r \vertiii{A_{S_i}^n}
=
\sharp \Sigma\times d^n + o(d^n)$.

$\ref{pleinpoints}\Rightarrow\ref{rhoAeqd}$ Hypothesis~$\ref{pleinpoints}$
means that~$\vertiii{A_S^n} = c\times d^n + o(d^n)$.
Hence (Gelfand theorem again):
$$
\rho(A_{S})
=
\lim_{n \to \infty} \root{n}\of {\vertiii{A_{S}^n}}
=
\lim_{n \to \infty} \root{n}\of {c\times d^n + o(d^n)}
= d.
$$

\medbreak

Along the proof of the equivalence, we have proved that the constant
must be equal to~$\sharp\Sigma$. 

For the last assertion, we remark that~$\rho(A_S) \leq d$:
that is because $A_S$ is a non-negative matrix whose sums on any line and
any column are between $0$ and $d$. If assertion~$\ref{rhoAeqd}$ is not
satisfied, this means that~$\rho(A_S) < d$. Using again
the Gelfand spectral radius theorem, one directly  proves
that~$\vertiii{A_S^n} =  o(d^n)$.
\end{preuve}

\begin{rk}
The matrix norm~$\vertiii{A}$ we use is not
submultiplicative. Fortunately, Gelfand spectral radius theorem does not require this assumption (cf.~\cite[Corollary~5.7.10]{MatrixAnalysis}).
\end{rk}

\subsection{The~$(\beta_r)_{r\geq 1}$ sequence of recursive towers}


We apply most of the results of the preceding sections to show a specific
property satisfied by recursive towers. Especially, we focus on
the~$(\beta_r)_{r\geq 1}$ sequence.

We need some common notations for the statements of this section.
Let $C$ be a singular curve defined over $\FF_q$,
let~$\nu : {\widetilde C} \rightarrow C$ denote the normalization map,
and let $P \in C(\FF_{q^r})$ be a point of $C$ defined over~$\FF_{q^r}$
for some~$r \in \NN^* \cup \{\infty\}$ (with the
convention~$\FF_{q^\infty} = \overline{\FF_q}$).
For~$s \in \NN^* \cup \{\infty\}$, we put
$$
\nu_P\left(\FF_{q^s}\right)
=
\sharp \left\{Q \in {\widetilde C}\left(\FF_{q^s}\right) \vert \nu(Q)=P\right\},
$$
the number of $\FF_{q^s}$-rational points of ${\widetilde C}$ above $P$.

We will make heavy use of the $r$-arithmetic graph~$\Graphe_r$ (see
definition~\ref{vocabulaire}), whose
adjacency matrix is denoted by~$A_r$.
Since the norm $\vertiii{A_r^n}$ is the number of paths of length~$n$ in the
arithmetic graph by~(\ref{norme}), we have
\begin{equation} \label{somme}
\sharp \widetilde{C}_{n+1}(\FF_{q^r})
=
\vertiii{A_r^n}
+
\sum_{\text{$P \in C_{n+1}(\FF_{q^r})$}} (\nu_P(\FF_{q^r})-1),
%
\end{equation}

We first give another necessary condition for a recursive tower to have at
least one non-zero~$\beta_r$.

\begin{prop}\label{d-reg_ou_plein_de_pts_dans_desingularisation}
Let~$(\Cun,\Cdeux)$ be a correspondence as in section~\ref{s_recursive_tower}.
Suppose that
the curves~$C_n$ of the associated tower are all irreducible, and that the geometric genus sequence~$(g_n)_{n\geq 1}$ goes to~$+\infty$. Suppose also that at least one $\beta_r$ is non-zero. Then:
\begin{enumerate}
\item \label{exists-d-reg} either the graph~$\Graphe_\infty(\Cun,\Cdeux)$ has exactly one finite $d$-regular strongly connected
component; 
\item \label{plein_de_pts_desingularisation} or the number of new geometric points in $\widetilde{C_n}(\overline{\mathbb F}_q)$ coming from desingularization has asymptotic behaviour:
$$\sum_{P \in C_{n+1}({\overline \FF}_q)} (\nu_P({\overline \FF}_q)-1) = c \times d^n + o(d^n)$$
for some constant $c>0$.
\end{enumerate}
\end{prop}

\begin{preuve}
Let $r \geq 1$ be such that $\beta_r \neq 0$. We use proposition~\ref{cle} with $S = \Cun(\FF_{q^r})$, hence $\Graphe_S=\Graphe_r$ is the $r$-th arithmetic
graph, whose adjacency matrix is $A_r$.
Thanks to~(\ref{somme}), we have
$$
\sharp \widetilde{C}_{n+1}(\FF_{q^r})
=
\vertiii{A_r^n}
+
\sum_{\text{$P \in C_{n+1}(\FF_{q^r})$}} (\nu_P(\FF_{q^r})-1),
$$
By lemma~\ref{exacte}, since~$\beta_r \not=0$, we know
that there exists a constant~$c > 0$ such
that~$\sharp \widetilde{C}_{n+1}(\FF_{q^r}) = c\times d^n + o(d^n)$.
Suppose that $\ref{exists-d-reg}$ does not hold. Then
by proposition~\ref{cle} (last statement), we have~$\vertiii{A_r^n} = o(d^n)$.
Therefore~\ref{plein_de_pts_desingularisation} must holds.
%
%
\end{preuve}

In the following result we show under some hypothesis that the
sequence~$(\beta_r)_{r\geq 1}$ of a recursive tower could have at most an unique
non-zero term.

\begin{theo} \label{betar}
Let~$(\Cun,\Cdeux)$ be a correspondence as in section~\ref{s_recursive_tower}. Suppose that the curves~$C_n$ of the associated tower are all irreducible
and that the geometric genus sequence~$(g_n)_{n\geq 1}$ goes to~$+\infty$.  Suppose also that the following hypothesis holds:

(H) the number of new geometric points of $\widetilde{C}_n$ coming from the desingularisation of $C_n$ is negligible compared to $d^n$ for large $n$, that is:
$$\sum_{\text{$P \in C_{n+1}({\overline \FF}_{q})$ }} (\nu_P({\overline \FF}_{q})-1) = o(d^n).$$
Then, there exists at most one integer $r \geq 1$ such that $\beta_r \neq 0$.
\end{theo}

\begin{rks}
\item The irreducibility and genus-behaviour hypotheses are quite cheap. 
\item Up to our knowledge, the only examples in the literature who do not satisfy the desingularization assumption (H) are the tower defined by the recursive equation $(x+1)^3=y^3+1$ over $\FF_4$ (\cite{GS-FFA} or  \cite[example 2.4]{Beelen}), and a tower of Shimura curves computed by Elkies (see \cite{ElkiesET}).
\item The conclusion of this theorem is false for good towers constructed from Hilbert class field towers using Grunwald-Wang theorem as communicated to us by  Philippe Lebaque. It is also false for pull back of recursive towers, see the last section \ref{several} at the end of this paper.
\end{rks}

\begin{preuve}
For~$s\geq 1$, we consider the 
$s$-th arithmetic graph~$\Graphe_s$ and
we denote by~$A_s$ its adjacency matrix. We have by~(\ref{somme}):
\begin{equation}
\sharp \widetilde{C}_{n+1}(\FF_{q^s})
=
\vertiii{A_s^n}
+
\sum_{\text{$P \in C_{n+1}(\FF_{q^s})$}} (\nu_P(\FF_{q^s})-1),
%
\end{equation}

Hypothesis~(H) implies that the sum on the right hand side is negligible compared to~$d^n$. As for the norm~$\vertiii{A_s^n}$, it depends whether or not
conditions of proposition~\ref{cle} are satisfied for
the finite set~$S=\Cun(\FF_{q^s})$.

Suppose that the graph~$\Graphe_\infty$ does not contain any finite $d$-regular
strongly connected component. Then, by proposition~\ref{cle},
for any~$s \geq 1$, one also has~$\vertiii{A_s^n} = o(d^n)$.
We deduce
that~$\sharp \widetilde{C}_n(\FF_{q^s}) = o(d^n)$ and it follows by
lemma \ref{exacte} that all $\beta_s$'s vanish.

Suppose now
 that the graph~$\Graphe_\infty$ admits at least one finite $d$-regular
strongly connected component. By theorem~\ref{uniquecomposympa},
this component is unique and we denote it by~$\Graphe_\Sigma$.
By proposition~\ref{cle},
either~$\Sigma\subset \Cun(\FF_{q^s})$
in which case~$\vertiii{A_s^n} = \sharp\Sigma \times d^n + o(d^n)$,
or~$\Sigma\not\subset \Cun(\FF_{q^s})$ in which case~$\vertiii{A_s^n} = o(d^n)$.
Due to hypothesis~(H), we deduce that, for $s \geq 1$:
\begin{equation*}
\Sigma\subset \Cun(\FF_{q^s})
\quad\Leftrightarrow\quad \sharp \widetilde{C}_n(\FF_{q^s}) = \sharp\Sigma \times d^n + o(d^n),
\end{equation*}
and
\begin{equation} \label{negligeable}
\Sigma\not\subset \Cun(\FF_{q^s})
\quad\Leftrightarrow\quad \sharp \widetilde{C}_n(\FF_{q^s}) = o(d^n).
\end{equation}

We denote by~$r \geq 1$ the smallest integer such
that $\Sigma \subset \Cun(\FF_{q^{r}})$.
Then we have~$\widetilde{C}_n(\FF_{q^r}) = \sharp\Sigma \times d^n + o(d^n)$
and
\begin{equation}\label{produit}
\frac{d^n}{g_n}
=
\frac{d^n}{\sharp\widetilde{C}_n(\FF_{q^r})}
\times
\frac{\sharp\widetilde{C}_n(\FF_{q^r})}{g_n}
=
\frac{1}{\sharp\Sigma+o(1)}
\times
\frac{\sharp\widetilde{C}_n(\FF_{q^r})}{g_n}.
\end{equation}
Since both right hand side sequences admit a limit, so do the left hand side one.
Let~$\ell = \lim_{n\to+\infty}\frac{d^n}{g_n}$. If this limit is
zero then by lemma \ref{exacte} every~$\lambda_i({\mathcal T})$, and thus
every~$\beta_i({\mathcal T})$, vanishes and the theorem is proved.
We suppose now that~$\ell > 0$. From (\ref{negligeable}) and~(\ref{produit}), one has for any $s \geq 1$:
\begin{equation} \label{lambdas}
\lambda_s({\mathcal T})
=
\lim_{n\to+\infty} \frac{\sharp\widetilde{C}_n(\FF_{q^s})}{g_n}
=
\lim_{n\to+\infty} \left(\frac{\sharp\widetilde{C}_n(\FF_{q^s})}{d^n}
\times \frac{d^n}{g_n}\right)
=
\begin{cases}
\sharp\Sigma \times \ell & \text{if~$\Sigma \subset \Cun(\FF_{q^s})$} \\
0 & \text{if~$\Sigma \not\subset \Cun(\FF_{q^s})$}.
\end{cases}
\end{equation}

We now conclude in three steps. First,
by minimality of~$r$, we have~$\Sigma \subset \Cun(\FF_{q^s})$
if and only if~$r \mid s$. In particular,
$$
\lambda_r({\mathcal T}) = \sharp\Sigma \times\ell
\qquad \text{and} \qquad
\forall s \mid r, \, s\not= r, \quad \lambda_s({\mathcal T}) = 0.
$$
The last vanishing implies that~$\beta_s({\mathcal T}) = 0$ for~$s$
strictly dividing~$r$, thus
\begin{equation} \label{lambdar}
\lambda_r({\mathcal T}) = r\beta_r({\mathcal T}) = \sharp\Sigma \times \ell
\end{equation}
hence $\beta_r({\mathcal T}) \not=0$.
Second, for~$k\geq 2$, $\Sigma\subset{\FF_{q^{kr}}}$, then thanks to~(\ref{lambdas}) and (\ref{lambdar}):
$$
\sum_{\genfrac{}{}{0pt}{}{d\mid kr}{d\not= r}} d\beta_d({\mathcal T}) = \lambda_{kr}({\mathcal T}) 
- r\beta_r({\mathcal T}) =  \sharp\Sigma \times \ell- \sharp\Sigma \times \ell = 0.
$$
Therefore~$\beta_s({\mathcal T}) = 0$ for every~$s$ strictly divisible by~$r$.
Last, for~$s$ which neither divides~$r$ nor is divisible by~$r$,
then~$\Sigma\not\subset\Cun(\FF_{q^s})$, hence by (\ref{negligeable}) and lemma~\ref{exacte}, we have~$\lambda_{s}({\mathcal T}) =0$
and, a fortiori, $\beta_{s}({\mathcal T}) = 0$. This concludes the proof.
\end{preuve}

Theorem~\ref{betar} is a good tool to compute the defect of a recursive
tower. As in section~\ref{s_computation_genus},
let us consider the BGS tower as an example.

\begin{cor} \label{GStower}
Consider Bezerra-Garcia-Stichtenoth's tower defined in section~\ref{s_computation_genus} over $\FF_q$. Then:
\begin{enumerate}
\item \label{beta3} one has
$$
\beta_3 = \frac{2(q^2-1)}{3(q+2)},
\qquad
\beta_r=0, \,  r \neq 3,
\qquad
\lambda_3 = \frac{2(q^2-1)}{(q+2)} \; ;
$$

\item\label{defect} the defect $\delta_{BGS}$ of this tower, as defined in \cite{TV02}, is given by
$$
\delta_{BGS} = 1 - \frac{2(q^2-1)}{(q+2)\sqrt{q^3-1}}  \; ;
$$

\item \label{zeta} the zeta function of this tower, as defined in \cite{TV02}, is given by
$$
Z_{BGS}(T) = \frac{1}{(1-T)^{\beta_3}}.
$$
\end{enumerate}
\end{cor}

\begin{rk}
Assertion $\ref{defect}$ states that for $q$ large, we have $\delta_{BGS} =  1-\frac{2}{\sqrt q} +o\left( \frac{1}{\sqrt q}\right)$, so that the tower is good, but far from being optimal in the sense of \cite{TV02} (see section \ref{s_invariants}).
\end{rk}

\begin{preuve}
  This tower is known to be irreducible from \cite{Ap3-2}. What proposition 3.1 in \cite{Ap3-3} states in our context is that some explicit set $\Omega \subset {\mathbb P}^1(\FF_{q^3})$ of cardinal $q(q+1)$ is forward complete and outside the singular graph $\Graphe_{sing}$. 
 Since it is finite, it is complete by lemma \ref{forward_complete_complete}. The graph~$\Graphe_\Omega$ is thus $d$-regular; it must also be strongly connected since otherwise there would be more than two strongly connected components in~$\Graphe_\infty$, contradicting theorem~\ref{uniquecomposympa}. Since $\Omega \subset  {\mathbb P}^1(\FF_{q^3})$ but $\Omega \nsubseteq  {\mathbb P}^1(\FF_{q})$, we have $\beta_3 \neq 0$.
 By proposition~\ref{prop_genre_Ap3}, assumption (H) in theorem~\ref{betar} holds true since $d=q$, hence there exists by theorem~\ref{betar} at most one $r$ such that $\beta_r$ is non-zero, i.e.  $\beta_r=0$ for $r \neq 3$.

Now, the number of rational points of $\widetilde{C}_n$ over $\FF_{q^3}$ is given by equation~(\ref{somme}) for $r=3$. Here, the sum is $o(q^n)$ since (H) holds true as already stated. As for the norm part, condition~$\ref{Sigmadreg}$ of proposition~\ref{cle} 
for $S={\mathbb P}^1(\FF_{q^3})$ holds true for $\Sigma=\Omega$.  Thus condition~$\ref{pleinpoints}$ holds also true for $c = \sharp \Omega$, so that~(\ref{somme}) for $s=3$ reduces to
$$\sharp \widetilde{C}_n(\FF_{q^3}) = \sharp \Omega \times q^n+o(q^n).$$
Finally, the genus of $\widetilde{C}_n$ is given in \cite{Ap3-3}, so that~$\ref{beta3}$ is proved.
Items~$\ref{defect}$ and~$\ref{zeta}$ follow immediately by definitions of $\delta_{BGS}$ and $Z_{BGS}(T)$.
\end{preuve}

%
%
%
%
%
\subsection{Several non-zero~$\beta_r$'s}\label{several}

Following~\cite{HessStichTutdere}, let us call {\em positive parameter} of a tower~${\mathcal T}$ an integer~$r \in \NN^*$ such
that~$\beta_r({\mathcal T}) \not=0$. Theorem~\ref{betar} states that,
under the desingularization hypothesis, the set of positive parameters
of a {\em recursive} tower contains at most one element.

Recently, this set has been studied by some authors
(see e.g.~\cite{Lebacque_2,Lebacque_1,HessStichTutdere,BalletRolland}).
They prove that there exist towers of function fields with a set of
positive parameters arbitrarily large. For example,
Hess, Stichtenoth and Tutdere \cite{HessStichTutdere} pullback a known good
{\em recursive tower} by a well chosen curve. More precisely,
let~$(\Cun,\Cdeux)$ be a correspondence as in section~\ref{s_recursive_tower},
let~${\mathcal T} = (C_n)_{n\geq 1}$ the associated tower.
Let $\pi : Y \rightarrow \Cun$ be
a finite surjective morphism, with $Y$ absolutely irreducible. We
define the {\em pullback}~$\pi^*{\mathcal T} = (D_n)_{n\geq 1}$
of the tower~${\mathcal T}$ by~$\pi$ by~$D_1 = Y$
and~$D_{n+1}=D_n\times_{C_n}C_{n+1}$. Let us point out that the
pullback tower may not be recursive in our sense. One way to convince ourselves
of this is precisely to remark
that this tower may have a set of positive parameters with more than one
element, while it is easily seen that if ${\mathcal T}$ satisfies the desingularization assumption, then so does $\pi^*({\mathcal T})$. We prove:

\begin{cor} \label{basechange}
Let~$(\Cun,\Cdeux)$ be a correspondence satisfying the hypothesis of theorem~\ref{betar}. Let $\pi : Y \rightarrow \Cun$ be a finite surjective morphism of degree $d$, with $Y$ absolutely irreducible.
Then the pullback tower $\pi^*({\mathcal T}(\Cun, \Cdeux))$
has a set of positive parameters with at most $d$ elements.
\end{cor}

\begin{preuve} Thanks to theorem~\ref{betar}, the recursive tower ${\mathcal T}(\Cun, \Cdeux)$ possesses either zero, or only one non-vanishing $\beta_r$.
Note first that for any $r \geq 1$, the trivial bound $N_r(\widetilde{D}_n) \leq d \times \left(\sum_{s \mid r} N_s(\widetilde{C}_n)\right)$ implies
\begin{equation}\label{Couillone}
0 \leq \frac{N_r(\widetilde{D}_n)}{g(\widetilde{D}_n)} \leq d \times \sum_{s \mid r} \frac{N_s(\widetilde{C}_n)}{g(\widetilde{C}_n)} \times \frac{g(\widetilde{C}_n)}{g(\widetilde{D}_n)}  \leq d \times \sum_{s \mid r} \frac{N_s(\widetilde{C}_n)}{g(\widetilde{C}_n)}.
\end{equation}

Suppose first that ${\mathcal T}(\Cun, \Cdeux)$ possesses no non-vanishing $\beta_s$: for any $s$, the sequence $\frac{N_s(\widetilde{C}_n)}{g(\widetilde{C}_n)}$ goes to zero as $n$ goes to infinity. Let $r \geq 1$. All terms of the right hand side in (\ref{Couillone}) goes to zero as $n$ goes to infinity, so that the pulled back tower satisfy
$\lambda_r(\pi^*({\mathcal T}(\Cun, \Cdeux)))=0$.

\medskip

Suppose now that $\beta_{r_0}({\mathcal T}(\Cun, \Cdeux))\neq 0$ (hence all other $\beta_r$ vanish). Let $r \geq 1$. 
In case $r$ is not divisible by $r_0$, one proves in the same way using formula (\ref{Couillone}) that $\lambda_r(\pi^*({\mathcal T}(\Cun, \Cdeux)))=0$. In case $r=k\times r_0$, we need the sharper bound 
$$
B_r(\widetilde{D}_n) \leq \sum_{s\mid r, \, sd \geq r} \frac{ds}{r} \times B_s(\widetilde{C}_n),
$$
which implies
\begin{equation}\label{PasCouillone}
0
\leq
\frac{B_r(\widetilde{D}_n)}{g(\widetilde{D}_n)}
\leq
\sum_{s\mid r, \, sd \geq r} \frac{ds}{r} \times \frac{B_s(\widetilde{C}_n)}{g(\widetilde{C}_n)} \times \frac{g(\widetilde{C}_n)}{g(\widetilde{D}_n)}
\leq
\sum_{s\mid r, \, sd \geq r}  \frac{ds}{r} \times \frac{B_s(\widetilde{C}_n)}{g(\widetilde{C}_n)}.
\end{equation}
In our case $r = k\times r_0$, suppose that $k >d$. Then any index $s$ in the right-hand sum of (\ref{PasCouillone}) will be strictly greater that $r_0$, hence the right hand side goes to zero as $n$ goes to infinity, so that the left hand side also. 
\end{preuve}

\bibliographystyle{amsalpha}
\providecommand{\bysame}{\leavevmode\hbox to3em{\hrulefill}\thinspace}
\providecommand{\MR}{\relax\ifhmode\unskip\space\fi MR }
\providecommand{\MRhref}[2]{%
  \href{http://www.ams.org/mathscinet-getitem?mr=#1}{#2}
}
\providecommand{\href}[2]{#2}

\bigbreak

\noindent
\begin{minipage}[t]{0.6\textwidth}
Hallouin Emmanuel ({\tt hallouin@univ-tlse2.fr})

Perret Marc ({\tt perret@univ-tlse2.fr})
\end{minipage}
\hfill
\begin{minipage}[t]{0.3\textwidth}
Universit\'e Toulouse~2

5, all\'ees Antonio Machado

31058 Toulouse cedex

France
\end{minipage}
\end{document}